\documentclass[8pt,reqno,letterpaper]{amsart} 
\usepackage{amsmath,amssymb,rawfonts}
\usepackage[all]{xy}
\usepackage[utf8]{inputenc} 
\usepackage[spanish]{babel}

\newtheorem{theorem}{Teorema}[section]
\newtheorem{lemma}[theorem]{Lema}
\newtheorem{principle}[theorem]{Principio}
\newtheorem{definition}[theorem]{Definición}
\newtheorem{proposition}[theorem]{Proposición}

\newtheorem{notation}[theorem]{Notación}
\newtheorem{comments}[theorem]{Comentario}
 
\newtheorem{remark}[theorem]{Observación}
\newtheorem{statement}[theorem]{Afirmación}

\begin{document} 
\title{Aritmética}
\dedicatory{En memoria de Cata...}
\author{Joel Torres Del valle}
\email{jrorresdv1@gmail.com}

\maketitle

\tableofcontents

\section*{Introducción}

\noindent  Durante mucho tiempo existió la pretensión de concebir las matemáticas como una idealización del mundo palpable, y proceder sobre ellas como se haría entre objetos del mundo real. A finales del siglo XIX comenzó la aparición de paradojas en la joven Teoría de Conjuntos del matemático ruso Georg Cantor (1845-1918) y, de esta manera, las matemáticas, que ostentaban el título no meritorio de un \emph{ciencia exacta}, comenzó a desvanecerse. Se observó, pues, que el bello edificio se encontraba parado sobre arenas movedizas y se tambaleaba, al son del viento más ligero.
\smallskip
 
\noindent   Comenzó entonces un programa de fundamentación de las matemáticas, con el cual se pretendía la construcción de cimientos sólidos sobre los cuales parar el edificio matemático. De las diversas corrientes de Filosofía matemática que abordaron el problema, podemos señalar al \emph{Formalismo}.  El programa formalista estuvo principalmente impulsado por el matemático alemán David Hilbert (1862-1943), y su requerimiento final era una prueba de la no contradicción de las matemáticas. Por supuesto, damos por descontado, el sueño de la completitud de los sistemas formales sobre los cuales se fundamentarían las matemáticas. Para David Hilbert, la lógica y las matemáticas son  teorías de forma y no de sentido \cite{campos}, \cite{diodonie}. Es decir, \emph{todas las pruebas se debían llevar a cabo mediante reglas fijas sobre el manejo de símbolos, sin tener en cuenta en ningún momento el significado de los mismos}. Pensamiento que queda en claro cuando este dice: \emph{la matemática es un juego con reglas muy sencillas, que dejan marcas sin significado sobre el papel.} 
\smallskip
 
\noindent   El afán por una prueba de la no contradicción viene luego de que a partir los trabajos de Gottlob Frege (1848-1925), {\it Begriffschrift, a formula language, modeled upon that of arithmetic, for pure thought} (University of Jena, 1879)  se pudiera deducir la paradoja de Russell, que en la simbología de Peano podría expresarse como\footnote{Tomada de la carta de Bertrand Russell (1872-1970) a Gottlob Frege en la que se plantea la paradoja.}

\[
w=\mathrm{cls}\cap x\backepsilon(x\sim \epsilon x).\supset:w\in w.=.w\sim\epsilon w.
\]

\noindent   La cual, en palabras del propio Russell, corresponde a:  {\it Let $w$ be the predicate: to be a predicate that cannot be predicated of itself. Can $w$ be a predicated of itself? From each answer its opposite follows. Therefore we must conclude that $w$ is not a predicate. Likewise there is no class (as a totality) of those classes which, each taken as a totality, do not belong to themselves.}\footnote{ Russell in letter to G. Frege.} Y no puedo dejar pasar, por supuesto, la altura con que Frege asume la aterradora noticia de que su trabajo permitía la aparición de paradojas y, lleno de humildad cientifica, añade una nota al segundo volumen de su trabajo  {\it Die Grundlagen der Arithmetik} (University of Jena, 1884) que se encontraba ya en imprenta, en la que comenta respecto al descubriemiento de Russell en el primer volumen  de su trabajo. 
\smallskip
 
\noindent   En el año 1910 aparece el primer volumen de \emph{Principia Mathematica}, en el que Bertrand Russell (1972-1970) y Alfred North Whitehead (1861-1947), construyen gran parte de las matemáticas sin paradojas ni contradicciones aparentes. Quedaba entonces la cuestión de si el sistema resultaba completo, habida cuenta que todo parecía indicar la consistencia. Los primeros pasos hacia un esclarecimiento de este requerimiento se dieron en una dirección que parecía  divisar una \emph{luz al final del túnel}. Primeramente en el año 1930, en su tesis doctoral \cite{godel} Kurt Gödel (1906-1978) demuestra la completitud del cálculo lógico de primer orden. Sin embargo, por esas paradojas de la suerte, es el mismo Gödel quien un año después, en 1931 demuestra que si a los Axiomas de Peano para la Aritmética, le juntamos toda la lógica de \emph{Principia Mathematica}, no obtenemos un sistema del cual se puedan demostrar todas las verdades sobre los números naturales \cite{godelii}. Concretamente\footnote{Ver: {\it Algunos resultados metamátematicos sobre completitud y consistencia}, Kurt Gödel, 1931. Puede encontrarse en \emph{Kurt Gödel: obras completas}, editado y traducido por Jesís Monsterín.} 

\begin{itemize}
\item[1.] El sistema $S$ no es completo, es decir, en él hay sentencias $\varphi$, tales que ni $\varphi$ ni $\neg\varphi$ son deducibles y, en especial, hay problemas indecidibles con la sencilla estructura $\exists xFx$, donde $x$ varia sobre los números naturales y $F$ es una propiedad de los números naturales.
\item[2.] Incluso si admitimos todos los medios lógicos de \emph{Principia Mathematica}  en la metamatemática no hay ninguna prueba de consistencia para $S$. Por consiguiente, una prueba de consistencia para el sistema $S$ solo puede llevarse a cabo con la ayuda de modos de inferencia que no esten formalizados en el sistema $S$.
\item[3.] Ni siquiera añadiendo a $S$ una cantidad finita de axiomas de tal forma que el sistema extendido permanezca siendo $\omega$-consistente.
\end{itemize}

\noindent   Sin embargo, estas frases inicialmente no eran algo que un matemático se preguntara, por ejemplo, un matemático se preguntaría ¿es la Conjetura de Goldbach cierta o falsa? etc. Sin embargo, mediante un proceso hoy conocido como Codificación de Gödel, Gödel construye una frase explicita, de la Teoría de números (pero que sin embargo no surge de manera natural en la misma) que no se puede demostrar ni refutar. Algo así como una formulación matemática de la paradoja del mentiroso: Gödel demuestra que los números naturales son lo suficientemente fuertes como para codificar todas las verdades del sistema $S$ y sobre el sistema S. Éste construye una sentencia explicita, digamos G que de ser demostrable implica su refutación y viceversa. De modo que ni G su negación son deducbibles si el sistema se supone consistente. Queda entonces la pregunta ¿cómo podemos hacer estas frases más matemáticas? En los años 70's  Jeff Paris (1944-Ahora) y Leo Harrington (1946-Ahora), en \cite{paris} muestran que el resultado combinatorio: 

\begin{quote}
Para todos números naturales $n, k, m$ hay un número $l$ tal que si $f:[n]^{n}\to k$, hay un $Y\subset l$ tal que $Y$ es homogéneo para $f$, $\mathrm{card}(Y)\geq m$, y si $y_{0}$ es el menor elemento de $Y$, entonces $\mathrm{card}(Y)\geq y_{0}$.
\end{quote}

\noindent   No es demostrable (aunque cierto) en PA (la Aritmética de Peano). De modo que se obtienen los primeros resultados de independecia ``matemáticos'' en Aritmética, como se deseaba. Antes, con las geometrías no-euclideanas ya se habían establecido ciertas independecias, como el \emph{V Postulado de Euclides}, etc. Además, en Teoría de conjuntos, se había establecido la independencia de la \emph{Hipótesis del continuo de Cantor,} por parte de Gödel, quién demostró que la Hipótesis no se podía refutar en ZFC, y Paul Cohen (1934-2007), quién probó que la negación de la Hipótesis tampoco se podría refutar en ZFC. Respecto a esta última podrían consultarse \emph{Set theory and the continuum hypothesis} (Dover Books on Mathematics) de Paul Cohen y el trabajo \emph{The Consistency of the Axiom of Choice and of the Generalized Continuum Hypothesis with the Axioms of Set Theory} (Princeton University press) de Kurt Gödel. Quiero terminar con esta metáfora del escritor argentino  Jorge Luis Borges (1899-986). 
\medskip

\begin{quote}
[...] Veinticinco símbolos suficientes (veintidós letras, el espacio, el punto, la
coma) cuyas variaciones con repetición abarcan todo lo que es dable expresar:
en todas las lenguas. El conjunto de tales variaciones integraría una
Biblioteca Total, de tamaño astronómico [...]
Todo estaría en sus ciegos volúmenes. Todo: la historia minuciosa del porvenir,
Los egipcios de Esquilo, el número preciso de veces que las aguas del
Ganges han reflejado el vuelo de un halcón, el secreto y verdadero nombre
de Roma, la enciclopedia que hubiera edificado Novalis, mis sueños y entresueños
en el alba del catorce de Agosto de 1934, la demostración del Teorema
de Pierre Fermat, los no escritos capítulos de Edwin Drood, esos mismos
capítulos traducidos al idioma que hablaron los garamantas, las Paradojas
de Berkeley acerca del tiempo y que no publicó, los libros de hierro de Urizen,
las prematuras epifanías de Stephen Dedalus que antes de un ciclo de
mil años nada querrían decir, el evangelio gnóstico de Basílides, el cantar que
cantaron las sirenas, el catálogo fiel de la Biblioteca, la demostración de la
falacia de ese catálogo. Todo, ...
\medskip

\hfill Jorge Luis Borges, \emph{La biblioteca Total}.
\end{quote}

\noindent Esta es una exposición de hechos sobre la Aritmética vista desde la lógica matemática. En la Sección 1 presentamos la aritmética de Peano, PA, y la teoría completa de $\mathbb{N}$, y mostramos que $\mathbb{N}$ es un modelo primo de la teoría de $\mathbb{N}$. En la Sección 2 nos ocupamos de los Teoremas de Incompletitud.  En la Sección 3 nos ocupamos de modelos no-estándar de la aritmética, en la Sección 4 presentamos el principio combinatorio de Paris-Harrington y en la Sección 5 su independencia. Los resultados aquí presentados, son citados de las referencias enlistadas al final.  Una discusión espistemológica e histótica de algunos aspectos discutidos aquí se puede encontrar en \cite{campos}, \cite{diodonie}. Sobre los Teoremas de Gödel con una aproximación más amena es fuente obligada el trabajo \emph{El Teorema de la incompeltitud de Gödel: versión para no iniciados} de Claudia Guitierrez (Revista Cubo Mat. Educ., Universidad de la Frontera, Vol. 1, 1999 pgs 68-75.). Finalmente, una discusión elegante, y prescisa sobre algunos aspectos antes mencionados se puede encontrar esparcida entre las líneas de \emph{Un paseo finito por el infinito} de Iván Castro y Jesús Pérez, editado por la Universidad Javeriana. Estas referencias, son realmente, un caramelo para intelecto.

\section{El modelo estándar}

\noindent Considere el conjunto de números naturales $\omega=\{0, 1, 2, ...\}$, en él queremos sumar, multiplicar y establecer un orden. Deseamos encontrar una descripción (teoría) para este conjunto, es decir, un conjunto de axiomas que nos permita establecer verdades acerca de $\omega$. Esta sección está dedicada a presentar el marco axiomático de Peano, PA,  y mostrar que todos los modelos de PA, tienen a $\mathbb{N}$ como subestructura elemental. Pricipalmente me baso en \cite{maker, kaye, keje}.

\begin{definition}\rm
Un \textbf{lenguaje formal de primer orden} $\mathcal{L}$, es una colección de símbolos para operaciones (símbolos de función), relaciones y, constantes.
\medskip

\noindent  Para nuestro interés, tomamos en consideración  $\mathcal{L}_{A}:=\{+, \cdot, <, \underline{0}, \underline{1}\}$ donde $+, \cdot$ son símbolos de función, $<$ es un símbolo de relación, y $\underline{0}, \underline{1}$ son constantes. 
\end{definition}

\noindent  Una \textbf{estructura} para $\mathcal{L}_{A}$, o una $\mathcal{L}_{A}$-estructura, es un objeto de la forma $\mathcal{M}=(M, +_{M}, \cdot_{M}, <_{M}, 0_{M}, 1_{M})$ donde $M$ es un conjunto no vacío, donde podemos \emph{interpretar} cada símbolo de $\mathcal{L}$.  La \textbf{estructura estándar} (o natural) de $\mathcal{L}_{A}$ es: $\mathbb{N}=(\omega, +, \cdot, <, 0, 1)$ donde $\omega=\{0, 1, 2, ...\}$ y $+, \cdot, <$ son la función suma, producto y la relación de orden lineal, respectivamente. Los \textbf{términos} de $\mathcal{L}_{A}$ son los símbolos constantes de éste, las variables y aquellos obtenidos por $+, \cdot$ aplicados a constantes y/o variables.

\begin{definition}\rm
Las \textbf{fórmulas} de $\mathcal{L}_{A}$, son definidas por inducción según se sigue: si $t_{1}, t_{2}$ son términos, entonces, $t_{1}=t_{2}, t_{1}<t_{2}$ son fórmulas. Si $\varphi, \psi$ son fórmulas, y $v$ es una variable, entonces, $\neg\varphi$, $\varphi\vee\psi$, $\varphi\wedge\psi$, $\forall v\varphi$ y $\exists v\varphi$ son fórmulas. 
\end{definition}

\begin{definition}\rm
Una \textbf{sentencia} es una fórmula cuyas varibles están todas ligadas a cuantificadores, en caso contrario se dice una \textbf{propiedad} o \textit{fórmula abierta}.  Una \textbf{teoría} es simplemente una colección de sentencias. Un \textbf{modelo} $\mathcal{M}$ para una teoría $T$ es una estructurta para su lenguaje en la cual son ciertos todos los axiomas de $T$, notamos $\mathcal{M}\models T$ para índicar que $\mathcal{M}$ es un modelo de $T$. Si $T$ tiene un modelo, se dice que es \emph{satisfacible}.
\end{definition}

\begin{definition}\rm Sea $\mathcal{M}$ una estructura, definimos $\mathrm{Th}(\mathcal{M}):=\{\varphi: \varphi$ es una sentencia y $\mathcal{M}\models\varphi\}$ donde $\mathcal{M}\models\varphi$ significa que $\varphi$ vale en $\mathcal{M}$. $\mathrm{Th}(\mathcal{M})$ recibe el nombre de \textbf{Teoría completa} de $\mathcal{M}$.
\end{definition}

\begin{comments}\rm
El nombre de {\it teoría completa} es intencional. Dada una teoría $T$ esta se dice completa siempre que para toda sentencia $\phi$ del lenguaje de $T$, y todo modelo $\mathcal{M}$ de $T$, $\mathcal{M}\models\phi$ o $\mathcal{M}\models\neg\phi$. Entonces $\mathrm{Th}(\mathcal{M})$ siempre es completa. Pareciera entonces que la busqueda de una descripción para $\mathbb{N}$ puede deternerse aquí pues tenemos una descripción `buena' en el sentido en que esta es completa. Sin embargo, dada una fórmula $\phi$ de $\mathcal{L}_{A}$ no siempre es fácil determinar cuando esta pertenece o no a $\mathrm{Th}(\mathbb{N})$. Además, por ser completa, no es recursivamente enumerable, por lo que nunca podremos conocerla del todo. Queremos una descripción un poco más simple, esta vendrá dada por PA. 
\medskip 

\noindent David Hilbert  era del pensamiento de que la no-contradicción de un conjunto de axiomas dados, nos da total derecho de pensar en la existencia de objetos (lo que hoy llamaríamos modelos) que cumplan dichos axiomas. De ahí su famosa frase: If the arbitrarily given axioms do not contradict each other through their consequences, then they are true, then the objects defined through axioms exist. That, for me, is the criterion of truth and existence. Este pensamiento termina de cimientarse con el nombrado Teorema de Completitud Gödel en el año 1930: {\it Una teoría es consistente si, y sólo si, es satisfacible} \cite{godel}.
\end{comments}

\begin{theorem}[Compacidad] \label{compacidad_theorem}
Para que un conjunto infinito de sentencias sea satisfacible, es necesario y suficiente que cada subconjunto finito suyo lo sea. 
\end{theorem}

\begin{proof}
Sea $\Delta$ un conjunto infinito de sentencias, sea $\varphi$ una contradicción deducida de $\Delta$, como las pruebas son finitas, existe $\Delta_{0}$ subconjunto de $\Delta$ finito tal que $\Delta_{0}\vdash\varphi$ (esto significa que hay una demostración de $\varphi$ a partir de $\Delta_{0}$), luego $\Delta_{0}$ es inconsitente. Si $\mathcal{M}\models\Delta$ entonces $\mathcal{M}\models\Delta_{0}$ para todo $\Delta_{0}\subset\Delta$ finito.
\end{proof}

\noindent Por el \emph{cardinal de una estructura}, entenderemos el cardinal de su universo. Tenemos el siguiente teorema:

\begin{theorem}[Löwenheim-Skolem-Tarski]\label{loweheim-skolem-tarski_theorem}
Si $T$ tiene por lo menos un modelo de cardinal infinito, entonces,  $T$ tiene un modelo de cualquier cardinal infinito. \hfill$\square$
\end{theorem}

\begin{definition}\rm
Un \textbf{homomorfismo} de una estructura $\mathcal{N}$  en $\mathcal{M}$, es una aplicación que conserva la interpretación de los símbolos del lenguaje $\mathcal{L}$, en cuestión. Una \textbf{inmersión} es un homomorfismo inyectivo. Se dice que $\mathcal{N}$ es una subestructura de $\mathcal{M}$ si el morfismo inmersión de $\mathcal{N}$ en $\mathcal{M}$ es una inmersión de estructuras.
\medskip

\noindent Una inmersión $\varphi:\mathcal{N}\to\mathcal{M}$ se dice \emph{elemental}, si para toda fórmula libre  en $n$ variables $\psi$ del lenguaje $\mathcal{L}$ se cumple que $\mathcal{N}\models\psi(\overline{a})$ si, y sólo si, $\mathcal{M}\models\psi(\varphi(\overline{a}))$ para todo $\overline{a}=a_{1}, ..., a_{n}\in N$. Se dice que $\varphi$ es \textbf{elemental}. Si el mapeo incluisión de $\mathcal{N}$ en $\mathcal{M}$ es una inmersión elemental, se dice que $\mathcal{N}$ es una \textbf{subestructura elemental de} $\mathcal{M}$.
\end{definition}

\subsection{La Aritmética de Peano, PA} La \textbf{Aritmética de Peano} (PA, por sus siglas en inglés, Peano arithmetic),
 es la $\mathcal{L}_{A}$-teoría engendrada en el seno de los axiomas siguientes:
leyes asociativas para $+$ y $\cdot$, sus elementos neutros $0$ y $1$, respectivamente,
distributividad, y los axiomas de orden lineal discreto para $<$ (orden total, hay un primer elemento 0, no hay mayor elemento, 
todo elemento tiene un sucesor, todo elemento diferente de 0
tiene un predecesor inmediato), 1 es el sucesor de 0, 
$x<y\to x+z<y+z$, y el esquema de inducción:
para toda $\mathcal{L}_{A}$ fórmula $\varphi(x, \overline{w})$,
tenemos el axioma

\[
\forall\overline{w}[\varphi(0, \overline{w})\wedge\forall x(\varphi(x, \overline{w})
\to\varphi(x+1, \overline{w}))\to\forall x\varphi(x, \overline{w})].
\]

\noindent $\mathbb{N}\models$PA, y claramente PA$\subset\mathrm{Th}(\mathbb{N})$. Así que todo modelo de $\mathrm{Th}(\mathbb{N})$ lo será de PA, pero no recíprocamente. 
\medskip

\noindent Observése que la inducción no es como tal un axioma, sino un proceso que determina un conjunto infinito de axiomas siempre que $\varphi(x, \overline{w})$ se reemplace por una fórmula particular del lenguaje $\mathcal{L}_{A}$. Antes de continuar, quiero citar una anotación de Gödel acerca de la incompletitud, hecha en \cite{godelii}.

\begin{quote}
\begin{small}
\noindent La verdadera razón para la incompletitud inherente en todo sistema formal de las matemáticas es qaue la formalización de tipos superiores se puede continuar en forma infinita... mientras que en cualquier sistema formal está disponioble sólo una cantidad numerable de ellos. Por esto, se puede mostrar que la proposición indecidible aquí contruida (en \cite{godelii}) se vuelve decidible siempre que se añadan tipos superiores apropiados (por ejemplo, el tipo $\omega$ al sistema de la Aritmética de Peano). Una situación análoga prevalece para el sistema axiomático de de la Teoría de Conjuntos.
\end{small}
\end{quote}

\noindent   Un par de resultados clásicos en PA, a saber, el Algoritmo de la división de Euclides y el Teorema de Bézout, son presentados a continuación.

\begin{theorem}[Algoritmo de Euclides]\label{algoritmo-euclides_theorem} Sea $M\models$PA y $a, b\in M$ con $a\neq0$. Entonces, existen   $r, s\in M$ únicos tales que \begin{equation}\label{al_euclides}
M\models(b=as+r\wedge r<a).
\end{equation}
\end{theorem}

\begin{proof}
La existencia de $r$ y $s$ la probaremos por inducción sobre $x$ en la fórmula 

\[
\exists r\exists s(x=as+r\wedge r>a).
\]

\noindent En efecto, la fórmula $(0=0a+0\wedge0<a)$ es cierta en $M$, ya que $a\neq0$. Supóngase ahora, por hipótesis de inducción que para $x, r, s$ en $M$, se tiene que 

\[
M\models(x=as+r\wedge r<a).
\]

\noindent Entonces, $M\models(x+1=as+(r+1))$, entonces, alguno $r+1=a$ o $r+1<a$ debe ocurrir, así, 

\[
M\models x+1=(a(s+1)+0);
\]

\noindent y en ambos casos tenemos: $M\models\exists r'\exists s'(x+1=as'+r'\wedge r'<a)$. Luego, por inducción, concluimos  $\eqref{al_euclides}$. Para probar la unicidad, supóngase que $b, b', s, s', r, r'$ están en $M$ con $M\models b=as+r=as'+r'\wedge r<a\wedge r'<a$. Si $s<s'$, entonces, $M\models b=as+r<a(s+1)\leq as'+r'=b$, y si $s'<s$ obtenemos una contradicción. Así, $s=s'$ y $M\models as+r=as+r'$, así $r=r'$. 
\end{proof}

\noindent Ahora, procedemos a dar definiciones de nociones usuales en Teoría de Números, formalizadas en el lenguaje de PA. 

\begin{definition}\rm
\hspace{0.05cm}
\begin{itemize}
\item[i)] \textit{Congruencia.} $x\equiv y$(mód $z$)$\leftrightarrow\left(z\neq0\wedge\left(\frac{x}{z}\right)=\left(\frac{y}{z}\right)\right)$.
\item[ii)] \textit{Primo.} $\mathrm{Prim}(x)\leftrightarrow(x\geq2\wedge\forall y\forall z(x|(yz)\to(x|y\vee x|z))$.
\item[iii)] \textit{Irreducible.} $\mathrm{Irred}(x)\leftrightarrow\forall y(y|x\to(y=1\vee y=x))$.
\item[iv)] \textit{Coprimos.} $(x, y)=1\leftrightarrow(x\geq1\wedge y\geq1\wedge\forall u(u| x\wedge u|y\to u=1))$.
\end{itemize}

\end{definition}

\begin{theorem}[Teorema de Bézout] Sea $M\models$PA, si $x, y\in M$ son coprimos, entonces existe $z\in M$ tal que $x$ tiene un inverso multiplicativo a saber, $z$, módulo $y$.  \hfill $\square$
\end{theorem} 

\noindent Una consecuencia importante del Teormea de Bézout es que en cualquier modelo de PA, las nociones de primo e irreducible, son equivalentes. Formalmente, para todo $x\in M\models\mathrm{PA}$, tenemos $\mathrm{PA}\vdash\mathrm{Prim}(x)\leftrightarrow\mathrm{Irred}(x)$ \cite{kaye}. Sea $T$ una  teoría. Sea $\mathcal{N}$ un modelo de $T$. Se dice que $\mathcal{N}$ es un \textbf{modelo primo} de $T$, si para todo $\mathcal{M}\models T$, $\mathcal{N}$ es una subestructura elemental de $\mathcal{M}$.

\begin{lemma}[Test de Tarski-Vaught]\label{test-tarski-vaught_lemma} Sea $\mathcal{M}$ una sub-estructura de $\mathcal{N}$. Entonces, $\mathcal{M}$ es elemental si, y sólo si, para cada fórmula $\psi(v, \overline{w})$ y $\overline{a}\in M^{n}$, si hay $b\in N$ tal que $\mathcal{N}\models\psi(b, \overline{a})$, entonces hay $c\in M$ tal que $\mathcal{N}\models\psi(c, \overline{a})$.   \hfill $\square$
\end{lemma}

\begin{theorem}\label{modelo-primo_theorem}
Sea $\mathcal{M}\models\mathrm{Th}(\mathbb{N})$, entonces, podemos ver a $\mathcal{N}$ como un segmento inicial de $\mathcal{M}$, esta inmersión es elemental, i.e., $\mathcal{N}$ es un modelo primo de $\mathrm{Th}(\mathbb{N})$.
\end{theorem}

\begin{proof}
 Sea $\psi(v, w_{1}, ..., w_{m})$ una $\mathcal{L}_{A}$-fórmula, y sean $n_{1}, ..., n_{m}\in \mathbb{N}$ tal que $\mathcal{M}\models\psi(v, \overline{n})$. Sea $\varphi$ la $\mathcal{L}_{A}$-sentencia

\[
\exists v\hspace{0.3cm}\psi(v, \underbrace{1+ \ldots +1}_{n_{1}-\mathrm{veces}}, \ldots, \underbrace{1+ \ldots +1}_{n_{m}-\mathrm{veces}}).
\]

\noindent Entonces, $\mathcal{M}\models\psi$ y $\mathbb{N}\models \psi$ ya que $\mathcal{M}\equiv\mathbb{N}$. Pero entonces, para algún $s\in\mathbb{N}$ tenemos que 

\[
\mathbb{N}\models\psi(s, \underbrace{1+ \ldots +1}_{n_{1}-\mathrm{veces}}, \ldots, \underbrace{1+ \ldots +1}_{n_{m}-\mathrm{veces}})
\]

\noindent y

\[
\mathbb{N}\models\psi(\underbrace{1+\ldots+1}_{s-\mathrm{veces}}, \underbrace{1+ \ldots +1}_{n_{1}-\mathrm{veces}}, \ldots, \underbrace{1+ \ldots +1}_{n_{m}-\mathrm{veces}}).
\]

\noindent Como la anterior es una $\mathcal{L}_{A}$-sentencia, 

\[
\mathcal{M}\models\psi(\underbrace{1+\ldots+1}_{s-\mathrm{veces}}, \underbrace{1+ \ldots +1}_{n_{1}-\mathrm{veces}}, \ldots, \underbrace{1+ \ldots +1}_{n_{m}-\mathrm{veces}}).
\]

\noindent y

\[
\mathcal{M}\models\psi(s, n_{1}, \ldots, n_{m}).
\]

\noindent Concluimos que $\mathbb{N}$ es un modelo primo de $\mathrm{Th}(\mathbb{N})$.
\end{proof}

\section{Incompletitud}

\noindent  En la sección anterior, dado el conjunto $\omega$ pensamos en cómo lograr una descripción de este conjunto en concordancia con las funciones y relaciones que allí nos interesa estudiar, a saber, la suma, el producto y el orden. Obviamente la descripción brindada por $\mathrm{Th}(\mathbb{N})$ es una descripción completa pero no es de total agrado tener una teoría con tantos axiomas. Ahora la pregunta sería si el sistema de PA sí es completo, la respuesta es No, y está dada en los Teoremas de la Incompletitud de Gödel. El proposito de esta sección es la introducción de dichos teoremas. 

\begin{definition}\rm La \textbf{clase de funciones recursivas parciales} $\mathcal{C}$ es la clase más pequeña de funciones  $f:A\to\mathbb{N}$ para algún $A\subset\mathbb{N}^{k}$ y algún $k\geq1$ tal que

\begin{itemize}
\item[i)]  $\mathcal{C}$ contiene a 0 y las \emph{funciones sucesor}: $0(x)=0, \forall x\in\mathbb{N}$ y $s(x)=x+1, \forall x\in\mathbb{N}$. Para cada $1\leq i\leq n\in\mathbb{N}$, $\mathcal{C}$ contiene las \emph{funciones proyección} $u_{i}^{n}(x_{1}, ..., x_{n})=x_{i}$. 
\item[ii)] Si $f(x_{1}, .., x_{k})$ y $g_{1}(\overline{y}), ..., g_{k}(\overline{y})$ están en $\mathcal{C}$, entonces  $h(\overline{y})=f(g_{1}(\overline{y}), ..., g_{k}(\overline{y}))$ está en $\mathcal{C}$ con la convención de que si $g_{1}(\overline{y}), ..., g_{k}(\overline{y})$ es indefinida, entonces así lo es $h$, o si  $f(g_{1}(\overline{y}), ..., g_{k}(\overline{y}))$ no lo es. 
\item[iii)] $\mathcal{C}$ es cerrado bajo \emph{recursión primitiva}, i.e., si $f(\overline{x})$ y $g(\overline{x}, y, z)$ están en $\mathcal{C}$, entonces así lo es $h(\overline{x}, y)$ definida por  $
h(\overline{x}, 0)=f(\overline{x})$ y  
\[
h(\overline{x}, y+1)=\left\{\begin{array}{ccc}
g(\overline{x}, y, h(\overline{x}, y))& & \\
\mathrm{indefinido}&\mathrm{si}&h(\overline{x}, y)\hspace{0.2cm}\mathrm{lo}\hspace*{0.2cm}\mathrm{es.}
\end{array}\right.
\]
\item[iv)] $\mathcal{C}$ es cerrado bajo, \emph{minimización}, i.e., si $g(\overline{x}, y)$ está en $\mathcal{C}$, entonces así lo es $h(\overline{x})=(\mu y)(g(\overline{x}, y)=0)$ definida por  
$h(\overline{x})=$ el menor $y$ tal que $g(\overline{x}, 0), g(\overline{x}, 1), ...., g(\overline{x}, y)$ están todas definidas y $g(\overline{x}, y)=0$. $h(\overline{x})$ indefinida, si no hay tal $y$.
\end{itemize}
  
\end{definition}

\begin{definition}\rm
La \textbf{clase de funciones primitivas recursivas} $\mathcal{PR}$ es la clase de funciones más pequeña  que tiene al 0, a $s$ y $u_{i}^{n}$ para cada $1\leq i\leq n$ y cerrado bajo \emph{composición y recursión primitiva}.  Las \textbf{funciones recursivas} son funciones totales $f:\mathbb{N}^{k}\to\mathbb{N}$ para cierto $k>0$ en $\mathcal{C}$.  Un conjunto $A\subset\mathbb{N}^{k}$ es \textbf{recursivo} si, y sólo si, lo es su \emph{función caracteristica}:

\[
\chi_{A}(\overline{x}):=\left\{\begin{array}{ccc}
1&\mathrm{si}&\overline{x}\in A\\
0&\mathrm{si}&\overline{x}\not\in A
\end{array}\right.
\]

\noindent es recursiva. Similarmente, $A$ es \textbf{recursivo}, si, y sólo si, $\chi_{A}\in\mathcal{PR}$.
\end{definition}
 
\noindent Un conjunto $A\subset\mathbb{N}^{k}$ es \textbf{recursivamente enumerable} si, y sólo si, $A$ es el dominio de alguna función recursiva parcial $f$, i.e., para todo $\overline{x}\in\mathbb{N}^{k}$ ($f(\overline{x})$ está definido $\Leftrightarrow \overline{x}\in A$).  Algunos ejemplos de funciones primitivas recursivas son: $+, \cdot, \max, \min$, etc.
 
\begin{definition}\rm
Una relación $R \subset\mathbb{N}^{n}$ es \textbf{primitiva recursiva} si su \emph{función característica} 
 
 \[
 \chi_{R}(x)=\left\{
 \begin{array}{ccc}
 0&\mathrm{si}&R(x)\\
 1&\mathrm{si}&\neg R(x)
 \end{array}\right.
 \]
 
\noindent es primitiva recursiva. 
\end{definition}

\begin{theorem}[Gödel, 1930] Cada problema de la forma $\forall xFx$ con $F$ recursiva primitiva es reducible a la cuestión de si una determinada fórmula de la lógica pura de primer orden es satisfacible o no (es decir, para cada $F$ recursiva primitiva podemos encontrar una fórmula de la lógica pura de primer orden, cuya satisfacibilidad es equivalente a la verdad de $\forall xFx$). \hfill $\square$
\end{theorem}

\noindent Como consecuencia de esta, Gödel prueba que, en particular, para PA, se cumple el teorema siguiente: 
\medskip

\begin{theorem}[Primer Teorema de Incompletitud, Gödel] Hay problemas de la lógica pura de predicados de primer orden, es decir,  fórmulas de la lógica pura de primer orden, respecto a las cuales no podemos probar ni su validez ni la existencia de un contraejemplo.  \hfill $\square$
\end{theorem}

\noindent  El anterior hecho da la impresión de entrar en contradicción con lo que el mismo Gödel prueba en 1930, y es que toda fórmula de la lógica de predicados de primer orden, o es válida o posee un contraejemplo. Sin embargo, lo que aquí se acota es que no siempre es posible \emph{demostrar} la existencia de dicho contraejemplo. No, por lo menos, para los sistemas formales que se han tomado en consideración en \cite{godelii}. 
\medskip

\noindent Para finalizar esta parte, acotamos este resultado debido a Gödel, del que se desprende, que supuesta al consistencia de PA, no podemos conseguir una prueba de esta dentro del mismo sistema. Y del que nos valdremos para probar que el Principio combinatorio de Paris-Harrington es independiente de PA.
\medskip

\begin{theorem}[Segundo Teorema de Incompletitud de Gödel] Sea $K$ una clase recursiva y consistente cualquiera de fórmulas. Entonces ocurre que la sentencia que dice que $K$ es consistente no es $K$-deducible.  \hfill $\square$
\end{theorem}

\subsection{Numeración de Gödel}

\begin{definition}\rm
Asignaremos a cada símbolo $s$ de nuestro lenguaje $\mathcal{L}_{A}$, un único número natural $\#(s)$ llamado \textbf{código de Gödel} de $s$. Los números se asignan de la manera siguiente: $\#('0')=1, \#('1')=2, \#('+')=3, \#('\cdot')=4, \#('=')=5, \#('(')=6, \#(')')=7, \#('\to')=8, \#('\neg')=9, \#('\forall')=10, \#('x_{i}')=11+i$. 
\end{definition}

\begin{definition}\rm
Sea $\Lambda\equiv\alpha_{1}...\alpha_{n}$ una fórmula (o término) del lenguaje $\mathcal{L}_{A}$, el código de Gödel de $\Lambda$, corresponde a $\#(\Lambda)=2^{\#('\alpha_{1}')}\cdot3^{\#('\alpha_{2}')}\cdot ...\cdot p_{n}^{\#('\alpha_{n}')}=\prod_{i=1}^{n}p_{i}^{\#('\alpha_{i}')}$, donde $p_{n}$ representa el $n$-ésimo primo. Por el Teorema fundamental de la Aritmética este código es único para cada fórmula (o término) del lenguaje $\mathcal{L}_{A}$.
\end{definition}

\noindent De la misma manera podemos códificar conjuntos finitos como números. Por ejemplo, sea $A=\{a_{1}, a_{2}, a_{3}, ...., a_{n}\}\subset\mathbb{N}$. Entonces, 

\[
\#('A'):=2^{a_{1}}\cdot3^{a_{2}}\cdot5^{a_{3}}\cdot...\cdot p_{n}^{a_{n}}:=\prod_{1\leq i\leq n}p_{i}^{a_{i}}.
\]

\noindent Así, en lugar de hablar del conjunto $A$, podemos hablar de su código, y en lugar de su cardinal, hablar de la cantidad de factores primos en el código de $A$. Obsérvese, que hay una forma natural de exprersar el hecho, '$p$ es primo' simbolicamente. A saber, 

\[
\mathrm{Prim}(p):=p\neq0\wedge p\neq1\wedge\forall x\leq p[x|p\to p=x\vee x=1].
\]

\noindent Con $x|p$ la  definida según la fórmula $x|y\leftrightarrow\exists z\leq y[xz=y]$. Ahora, vamos a observar la expresión '$p_{n}$ es el $n$-ésimo primo'. Sea $g(y, x)$ una función. Definimos el $\mu$-\textbf{operador acotado} por: $f(x, x)=\mu y<x[g(y, x)=0]$ a ser $f(x, x)=$ el menor $y<x$ tal que $g(y, x)=0$, si tal $y$ existe; y $f(x, x)=x$ en otro caso. Así, podemos definir: 

\begin{center}
$p_{n}:=n$-ésimo primo: $p_{0}=2$, $p_{n+1}:=\mu x<p_{n}!+1[p_{n}<x\wedge \mathrm{Prim}(x)]$.
\end{center}

\begin{definition}\rm

\[
a\in \mathrm{Seq}\leftrightarrow a=1\vee (a>1\wedge\forall x\leq a[p_{x+1}|a\to p_{x}|a]).
\]

\[
\mathrm{Long}(a):=\left\{
\begin{array}{ccc}
0&\mathrm{si}&a\not\in\mathrm{Seq}\vee a=1\\
\mu x\leq a[p_{x}|a\wedge \neg(p_{x+1}|a)]&\mathrm{si}&a\in\mathrm{Seq}\wedge a\neq 1.
\end{array}\right.
\]

\[
(a)_{x}:=\mu y\leq x+1[p_{x}^{y+1}|a\wedge \neg(p_{x}^{y+2}|a)].
\]

\noindent Seq denota el conjunto de números que son secuencias. Y Long($a$) la longitud de $a$. 
\end{definition}

\begin{remark}\rm

\[
a=\prod_{i\leq\mathrm{Long}(a)}p_{i}^{(a)_{i}+1}.
\]

\end{remark}

\begin{definition}\rm Una fórmula $\varphi$ es $\Sigma_{n}$ (resp. $\Pi_{n}$) si, y sólo si, para alguna fórmula recursiva primitiva $\psi$, $\varphi=Q_{1}x_{1}...Q_{n}x_{n}\psi$, donde $Q_{1}=\exists$ (resp. $\forall$) y los cuantificadores se alternan en el tipo y son todos acotados. Escribimos $\varphi\in\Sigma_{n}$ (resp. $\Pi_{n}$) para índicar que $\varphi$ es una $\Sigma_{n}$ (resp. $\Pi_{n}$) fórmula, o podemos demostrar que es equivalente  a una de estas.  
\end{definition}

\begin{remark}\rm
Gödel introdujo una \textbf{Función par} $\langle, \rangle$ la cual asigna a cada $(x, y)\in\mathbb{N}\times\mathbb{N}$ un único número natural. Esta será útil para codificar particiones de conjutnos finitos en el Capítulo 4. La Función par de Gödel está dada por

\[
\begin{array}{cccc}
\langle, \rangle:&\mathbb{N}\times\mathbb{N}&\longrightarrow&\mathbb{N}\\
 & (x, y)&\longmapsto&\frac{(x+y)(x+y+1)}{2}+y.
\end{array}
\]
\end{remark}

\begin{lemma}\rm
Para cualquier cuatro números naturales dados  $x, y, u, v$ se cumple lo siguiente $\langle x, y\rangle=\langle u, v\rangle$ si, y sólo si, $x=y$ y $u=v$. \hfill $\square$
\end{lemma}

\noindent Ahora, introducimos  los principios de reflexión. Asumimos que el conjunto de códigos de los axiomas de $T$ (un sistema formal dado, por ejemplo uno que contenga PA) es recursivo primitivo. Así, tenemos:
\bigskip

\begin{small}
$\mathrm{Prov}_{T}(x, y)\leftrightarrow x\in\mathrm{Seq}\wedge\forall i\leq\mathrm{long}(x)[(x)_{i}$ es un axioma lógico

\hfill $\vee(x)_{i}$ es un axioma de $T$ $\vee\exists j k<i((x)_{k}=\mathrm{imp}((x)_{j}, (x)_{i}))]\wedge y=(x)_{\mathrm{long}(x)}$.

\[
\mathrm{Prov}_{T}(y)\leftrightarrow\exists x\mathrm{Prov}_{T}(x, y).
\]
\end{small}

\noindent Es decir, $\mathrm{Prov}_{T}(x, y)$ afirma que $x$ es el número de Gödel de una demostración de $y$ en $T$. Y $\mathrm{Prov}_{T}(y)$ afirma que $y$ es demostrable en $T$. Suele notarce el código de Gödel de una fórmula $\varphi$ de la forma siguiente: $\ulcorner\varphi\urcorner$.  Sea $\varphi$ una sentencia. Entonces, el Lema de Löb establece que  $T\vdash\mathrm{Pr}_{T}(\ulcorner\varphi\urcorner)\to\varphi$ si, y sólo si, $T\vdash\varphi$.

\begin{principle}[Reflexión local, Rfn($T$)]\rm Sea $\varphi$ una sentencia, $\mathrm{Pr}_{T}(\ulcorner\varphi\urcorner)\to\varphi.$
\end{principle}

\begin{principle}[Reflexión uniforme I,  RFN($T$)]\rm Sea $\varphi$ una fórmula con la sóla  variable libre $x$.  

$$\forall x\mathrm{Pr}_{T}(\ulcorner\varphi(x)\urcorner)\to\forall x\varphi(x).$$
\end{principle}

\begin{principle}[Reflexión uniforme II, RFN'($T$)]\rm Sea $\varphi$ un fórmula con la sóla variable libre  $x$. 

$$\forall x[\mathrm{Pr}_{T}(\ulcorner\varphi(x)\urcorner)\to\varphi(x)].$$
\end{principle}

\begin{theorem}\rm
Sobre $S$, los siguientes son equivalentes: 

\begin{itemize}
\item[i)] $\mathrm{Con}_{T}$,
\item[ii)] $\mathrm{Rfn}_{\prod_{2}}(T)$,
\item[iii)] $\mathrm{RFN}_{\prod_{1}}(T)$,
\item[iv)] $\mathrm{RFN'}_{\prod_{1}}(T)$,
\end{itemize}

\noindent  donde el segundo subíndice $\prod_{1}$ índica la restricción de la elección a $\varphi\in\prod_{1}$.   \hfill $\square$
\end{theorem}

\begin{remark}\rm
Sea $\mathrm{RFN}_{\prod_{k}}(T)$ la restricción de fórmulas en $T$ en $\prod_{k}$. Similarmente, se define $\mathrm{RFN}_{\sum_{k}}(T)$, RFN'$\sum_{k}(T)$ y RFN'$_{\prod_{k}}(T)$.
Finalmente, la noción de $\omega$-\textbf{consistencia} es aquella dada a nuestra razón, según la cuál se cumple pára $T$, siempre que se cumplan para $T$, las dos condiciones siguientes: $T\vdash\exists x\varphi(x)$, $T\vdash\neg\varphi(\overline{0}), \neg\varphi(\overline{1}), ...$.
\end{remark}

\begin{theorem}[Primer Teorema de Incompletitud de Gödel] Sea $T\vdash\varphi\leftrightarrow\neg\mathrm{Pr}_{T}(\ulcorner\varphi\urcorner)$. Entonces:

\begin{itemize}
\item[i)] $T\nvdash\varphi$, 
\item[ii)] bajo un supuesto adicional, $T\nvdash\neg\varphi$.
\end{itemize} 
\end{theorem}

\begin{theorem}[Segundo Teorema de Incompletitud de Gödel] Sea $\mathrm{Con}_{T}$ igual a $\neg\mathrm{Pr}_{T}(\ulcorner\Lambda\urcorner)$, donde $\Lambda$ es cualquier afirmación contradictoria conveniente. Entonces, $T\nvdash\mathrm{Con}_{T}$.   \hfill $\square$
\end{theorem}

\noindent ¿Será entonces que Russell tenía razón? ¿Será que efectivamente las matemáticas pueden ser definidas como aquel tema del cual no sabemos nunca lo que decimos ni si lo que decimos es verdadero? Si bien no es un `nunca', es `muchas veces', la expresión que debemos usar. Debemos concluir, sin embargo que Hilbert estaba equivocado. En matemáticas sí hay ignorabimus. Gödel, quién estuvo más cerca que nadie de llevar a feliz término el programa formalista, fue quien ¡oh sorpresa! notó  que éste era irrealizable, a lo que muchos matemáticos respondieron con rechazo, como si desconocer la prueba evitaría la veracidad del teorema.

\section{Modelos no-estándar de la Aritmética}

\subsection{Existencia de Modelos no-estándar}\hspace{0.1cm}\\ Tenemos dos descripciones dadas para $\mathbb{N}$, a saber, PA y $\mathrm{Th}(\mathbb{N})$, querríamos saber si estas descripciones son univocas, es decir, si todo modelo de PA o de $\mathrm{Th}(\mathbb{N})$ son isomorfos a $\mathbb{N}$, la respuesta es No. Vamos a demostrar la existencia de modelos de $\mathrm{Th}(\mathbb{N})$ (y por lo tanto de PA)  muy ``parecidos'' a $\mathbb{N}$ pero no isomorfos a él. Tales modelos se conocen como \textbf{modelos no-estándar}.
\medskip

\noindent  La prueba de la existencia de dichos modelos se debe a Thoralf Skolem en el año 1934 en el trabajo titulado  {\it Über die Nicht-charakterisierbarkeit der Zahlenreihe mittels endlich oder abzählbar unendlich vieler Aussagen mit ausschließlich Zahlenvariablen}. Fundamenta Mathematicae, Alemania,  Vol. 23, No. 1, 150—161, 1934.

\begin{definition}\rm Considere la $\mathcal{L}_{A}$-teoría $\mathrm{Th}(\mathbb{N})$. Sea $n\in\mathbb{N}$. Definimos el $\mathcal{L}_{A}$-término $\underline{n}$ como 

\[
\underbrace{(...(((1+1)+1)+1)...+1)}_{n-\mathrm{veces}\hspace{0.03cm}1}
\]

\noindent y cero es simplemente el símbolo constante 0. Sea $c$ un nuevo símbolo constante. Consideremos $\mathcal{L}_{c}$ el lenguaje obtenido al unir a $\mathcal{L}_{A}$ el símbolo constante $c$. Sea $T$  la $\mathcal{L}_{c}$-teoría engendrada en el seno de los axiomas: 

\begin{center}
$\sigma$ (para cada $\sigma\in\mathrm{Th}(\mathbb{N})$)
\end{center}

\noindent y $c>\underline{n}$ (para cada $n\in\mathbb{N}$).
\end{definition}

\begin{remark}\rm  Por el Teorema de Compacidad, demostrar que $T$ es satisfacible es equivalente a demostrar que es finitamente satisfacible. Sea $\Delta_{0}\subset T$ finito. Entonces, existe $k\in\mathbb{N}$ tal que $\Delta_{0}\subset T_{k}\subset T$, con 

\[
T_{k}=\mathrm{Th}(\mathbb{N})\cup\{c>\underline{n}:n>k\}
\]

\noindent y evidentemente $(\mathbb{N}, k)\models T_{k}$. Luego, $T$ es finitamente satisfacible, y así, satisfacible.   Sea $\mathcal{M}_{c}\models T$. Ahora, como 

\[
\mathcal{M}_{c}\models c>\underline{n}
\]

\noindent para todo $n\in\mathbb{N}$, entonces, existe un entero ``infinito". Ahora, claramente $\mathcal{M}_{c}\models\mathrm{Th}(\mathbb{N}), PA$ en virtud de que $PA\subset\mathrm{Th}(\mathbb{N})\subset T$. 
\bigskip

\noindent Ahora vamos a mirar algunas propiedades interesantes de estos modelos no-estándar cuya existencia acabamos de demostrar. 
\end{remark}

\begin{proposition} El reducto (restricción) $\mathcal{M}$ de $\mathcal{M}_{c}$ al leguaje $\mathcal{L}_{A}$ no es isomorfo a $\mathbb{N}$. 
\end{proposition}
 
\begin{proof}
Supóngase que $h:\mathbb{N}\to M$ es un isomorfismo. Necesariamente, $h$ envía a $n\in\mathbb{N}$ en $\underline{n}^{\mathcal{M}}$ el elemento que realiza el término cerrado $\underline{n}$ en $\mathcal{M}$. Además, 

\[
\mathcal{M}\models \forall x\forall y(x>y\to\neg (x=y)),
\]

\noindent puesto que está sentencia es cierta en $\mathbb{N}$, luego está en $\mathrm{Th}(\mathbb{N})$. De aquí se sigue que el elemento que realiza a $c$ en $\mathcal{M}_{c}$ no es la imagen de $h$.
\end{proof}
 
\noindent Por el Teorema de Löweheim-Skolem-Tarski, podemos tomar estos modelos no-estándar con un cardinal infinito tan grande como queramos. Lo que nos dice que hay ``muchos" modelos no isomorfos de $\mathrm{Th}(\mathbb{N})$. 
\bigskip

\begin{proposition}  El mapeo $h:\mathbb{N}\to\mathcal{M}$ que envía a $n\in\mathbb{N}$ en $\underline{n}^{\mathcal{M}}$ es una $\mathcal{L}_{A}$-inmersión de estructuras.
\end{proposition}
 
\begin{proof}
 Para ver que $h$ es uno-uno, notése que si $n, k\in\mathbb{N}$ con $k\neq n$, entonces $ \mathbb{N}\models\neg(\underline{n}=\underline{k})$ luego la sentencia $\neg(\underline{n}=\underline{k})$ está en $\mathrm{Th}(\mathbb{N})$ por tanto  es cierta en $\mathcal{M}$. Similarmente $h$ preserva $<, +$ y $\cdot$, ya que para cualquier $k, n, m\in\mathbb{N}$, 
 
 \[
 n<m\Leftrightarrow\mathbb{N}\models \underline{n}<\underline{m}\Leftrightarrow \mathcal{M}\models\underline{n}<\underline{m},
 \]
 
 \[
 n+m=k\Leftrightarrow\mathbb{N}\models\underline{n}+\underline{m}=\underline{k}\Leftrightarrow \mathcal{M}\models \underline{n}+\underline{m}=\underline{k},
 \]
 
\noindent  y 
 
 \[
 n\cdot m=k\Leftrightarrow\mathbb{N}\models\underline{n}\cdot\underline{m}=\underline{k}\Leftrightarrow \mathcal{M}\models\underline{n}\cdot\underline{m}=\underline{k}.
 \]
\end{proof}

\begin{definition}\rm
Considere el lenguaje $\mathcal{L}_{A}$ de la aritmética y $PA$ los axiomas de Peano. Suponga que $\mathcal{M}, \mathcal{N}\models PA$. Decimos que $\mathcal{N}$ es una \textbf{extensión cofinal} de $\mathcal{M}$ si $M\subset N$ y $a<b$ para todo $a\in M$ y $b\in N-M$.
\end{definition}

\begin{remark}\rm Podemos siempre identificar a $\mathbb{N}$ como la imagen de $h$ en $\mathcal{M}$. Así, $\mathbb{N}$ es una subestructura de todo $\mathcal{M}\models\mathrm{Th}(\mathbb{N})$. De donde $\mathcal{M}$ es no-estándar si, y sólo si,  existe $a\in M$ tal que $a$ no es ningún $n\in\mathbb{N}$ estándar. A tales $a$'s se les conoce como \textbf{enteros no-estándar}.  El orden $<$ es un orden lineal sobre $M$ con menor elemento 0 y sin mayor elemento (esto se puede expresar en una sentencia de primer orden $\gamma$ que vale para $\mathbb{N}$, luego está en $\mathrm{Th}(\mathbb{N})$ de donde vale en $\mathcal{M}$). Ahora bien, sea $k\in\mathbb{N}$, tenemos que   

\[
\mathbb{N}\models\forall x\left(x<k\to\left(\bigvee_{i=0}^{k-1} x=\underline{i}\right)\right),
\]

\noindent luego la sentencia $\forall x\displaystyle\left(x<k\to\left(\bigvee_{i=0}^{k-1}x=\underline{i}\right)\right)$ está en $\mathrm{Th}(\mathbb{N})$ y así

\[
\mathcal{M}\models\forall x\left( x<k\to\left(\bigvee_{i=0}^{k-1} x=\underline{i}\right)\right).
\]

\noindent Así, ningún entero no-estándar vive por debajo de algún entero estándar. De donde, se obtiene que $\mathbb{N}$ es un \emph{segmento inicial} de $\mathcal{M}$ y este último una \emph{extensión cofinal} de $\mathcal{M}$. 
\end{remark}

\begin{proposition}
Sea $\theta(x)$ una $\mathcal{L}_{A}$-fórmula con una única variable libre  $x$. Sea $\mathcal{M}\models\mathrm{Th}(\mathbb{N})$ no-estándar. Entonces, existe un $a\in M$ no-estándar tal que $\mathcal{M}\models\theta(a)$ si, y sólo si, hay infinitos enteros estándar $k\in \mathbb{N}$ que cumplen $\mathbb{N}\models\theta(\underline{k})$.
\end{proposition}

\begin{proof}
 Supóngase que para algún $a\in M$  se tiene que $\mathcal{M}\models\theta(a)$. Entonces, $\mathcal{M}\models\exists x\theta(x)$ luego así es para $\mathbb{N}$ ya que si $\mathbb{N}\models\neg\exists x\theta(x)$ entonces, $\neg\exists x\theta(x)\in\mathrm{Th}(\mathbb{N})$ contradiciendo el supuesto. Como $\mathcal{M}\models\mathrm{Th}(\mathbb{N})$, entonces, $M\models a>\underline{n}$ ya que $a$ es no-estándar. Así, 

\[
\mathcal{M}\models\exists x(\theta(x)\wedge x>\underline{n}).
\]

\noindent Así, $\mathbb{N}\models\exists x(\theta(x)\wedge x>\underline{n})$. Se sigue que hay infinitos $k$ en $\mathbb{N}$ que cumplen $\mathbb{N}\models\theta(\underline{k})$. Reciprocamente, suponga que hay infinitos $k\in\mathbb{N}$ que cumplen $\mathbb{N}\models\theta(\underline{k})$ y suponga que $\mathcal{M}\models\mathrm{Th}(\mathbb{N})$ es no-estándar. Entonces, ya que $\mathbb{N}\models\forall x\exists y(y>x\wedge\theta(y))$, tenemos $\mathcal{M}\models\forall x\exists y(y>x\wedge(y))$. Así, para cualquier $b\in M$, y en particular $b\in M$ no-estándar, hay $a>b$ en $M$ tal que $\mathcal{M}\models\theta(a)$.
\end{proof}

\begin{theorem}
Existen  $2^{\aleph_{0}}$ modelos no isomorfos de $\mathrm{Th}(\mathbb{N})$. 
\end{theorem}

\begin{proof}
\rm Ver \cite{kaye}. 
\end{proof}

\subsection{$\mathbb{Z}$-cadenas} \hspace{0.1cm}\\

\noindent  Sea $\mathcal{M}\models\mathrm{Th}(\mathbb{N})$ no-estándar. Sea $c\in M$ no estándar. Entonces, los elementos $c-n$ y $c+n$ existen en $M$, para todo $n\in\mathbb{N}-\{0\}$. En efecto, la sentencia $\forall x(x\neq0\to\exists y(x=y+n))$ está en $\mathrm{Th}(\mathbb{N})$ para $n$ fijo en $\mathbb{N}-\{0\}$. Por lo tanto $c-n$ existe en $M$. De la misma forma $\forall x\exists y(x+n=y)\in\mathrm{Th}(\mathbb{N})$, para $n\in\mathbb{N}-\{0\}$ fijo; y así $c+n$ existe en $M$ para todo $n$ estándar diferente de 0. Así, motivamos la siguiente definición.

\begin{definition}\rm Sean $\mathcal{M}\models\mathrm{Th}(\mathbb{N})$ y $c\in M$ no estándar. El conjunto $\mathbb{Z}(c)=\{c\}\cup\{c-n, c+n:n\in\mathbb{N}-\{0\}$ se llama $\mathbb{Z}$-\textbf{cadena} asociada a $c$.
\end{definition}

\begin{remark}\rm 
Sean $\mathbb{Z}(e)$ y $\mathbb{Z}(c)$, $\mathbb{Z}$-cadenas. Se cumple una y sólo una se las siguientes afirmaciones: $\mathbb{Z}(d)=\mathbb{Z}(e)$ o $\mathbb{Z}(d)\cap\mathbb{Z}(e)=\emptyset$. En efecto, por el hecho de que $\mathbb{N}\models\forall x\forall y(x=y\vee x<y\vee x>y)$, tenemos que $\forall x\forall y(x=y\vee x<y\vee x>y)\in\mathrm{Th}(\mathbb{N})$. Ahora, si $d=e$, se tiene que $\mathbb{Z}(d)=\mathbb{Z}(e)$. Si $d\neq e$, supóngase (sin perdida de generalidad) que $d<e$, luego, si hay un número estándar $n$, de modo que $d+n=e$, en cuyo caso $\mathbb{Z}(d)=\mathbb{Z}(e)$, si no hay tal $n$, entonces $\mathbb{Z}(d)\cap\mathbb{Z}(e)=\emptyset$.
\end{remark}

\begin{notation}\rm  Notaremos $\mathbb{Z}(d)<\mathbb{Z}(e)$ para índicar que todo $x\in\mathbb{Z}(d)$ es menor que todo $x\in\mathbb{Z}(e)$. Así, es evidente que una (y sólo una) de las tres afirmaciones siguientes se cumple: $\mathbb{Z}(d)=\mathbb{Z}(e)$, $\mathbb{Z}(d)<\mathbb{Z}(e)$ y $\mathbb{Z}(e)<\mathbb{Z}(d)$. 
\end{notation}

\begin{lemma} Para toda $\mathbb{Z}$-cadena $\mathbb{Z}(d)$, existen  $\mathbb{Z}$-cadenas $\mathbb{Z}(d)$ y $\mathbb{Z}(c)$ de modo que $\mathbb{Z}(c)<\mathbb{Z}(d)<\mathbb{Z}(e)$.
\end{lemma}

\begin{proof}
 Basta considerar $e=d+d$. En efecto, $\forall x\forall y\forall z(x+y=x+z\to y=z)\in\mathrm{Th}(\mathbb{N})$. Si $d+d=d+n$, para $n\in\mathbb{N}$, nos da $d=n$ y esto no es posible puesto que $d$ no es estándar y $n$ sí. Luego, $d+d\not\in\mathbb{Z}(d)$. De aquí el argumento, pues $d+d>d$. Así, hemos conseguido una $\mathbb{Z}$-cadena $\mathbb{Z}(e)$ tal que $\mathbb{Z}(d)<\mathbb{Z}(e)$.
\medskip

\noindent La sentencia $\forall x(x\neq0\wedge x\neq 1\to\exists y(y<x\wedge(y+y=x\vee y+y=x+1)))$ vale en $\mathbb{N}$, luego está en $\mathrm{Th}(\mathbb{N})$. Luego, hay un $c$ tal que o bien $c+c=d$ o bien $c+c=d+1$, en todo caso $\mathbb{Z}(c)\neq\mathbb{Z}(d)$ y $\mathbb{Z}(d)<\mathbb{Z}(c)$ es imposible, de modo que $\mathbb{Z}(c)<\mathbb{Z}(d)$.
\medskip

\noindent En conclusión, hemos obtenido $\mathbb{Z}$-cadenas $\mathbb{Z}(e)$ y  $\mathbb{Z}(c)$ de modo que  $\mathbb{Z}(c)<\mathbb{Z}(d)<\mathbb{Z}(e)$.
\end{proof}

\begin{lemma}
Sean $\mathbb{Z}(d)$ y $\mathbb{Z}(e)$, $\mathbb{Z}$-cadenas diferentes, tales que $\mathbb{Z}(d)<\mathbb{Z}(e)$. Existe una $\mathbb{Z}$-cadena $\mathbb{Z}(f)$ tal que que $\mathbb{Z}(d)<\mathbb{Z}(f)<\mathbb{Z}(e)$.
\end{lemma}

\begin{proof}
La sentencia $\forall x\forall y\exists z(z+z=x+y\vee z+z=x+y+1)$ vale en $\mathbb{N}$, así que está en $\mathrm{Th}(\mathbb{N})$. Sea $f+f=d+e$ si $d+e$ es par, y sea $f+f=d+e+1$ si $d+e$ es impar. En cualquier caso $\mathbb{Z}(c)<\mathbb{Z}(d)<\mathbb{Z}(e).$ 
\end{proof}

\noindent De los Lemas anterores se sigue que \emph{el conjunto de $\mathbb{Z}$-cadenas es un orden lineal denso}. 

\begin{theorem}
Supóngase que para $\mathcal{M}\models\mathrm{Th}(\mathbb{N})$ el conjunto de $\mathbb{Z}$-cadenas es numerable. Entonces, al olvidar la estructura interna de cada $\mathbb{Z}$-cadena, el conjunto de $\mathbb{Z}$-cadenas es isomorfo a $\mathbb{Q}$, el conjunto de los números racionales. 
\end{theorem}

\begin{remark}\rm Si en la $\mathbb{Z}$-cadena $\mathbb{Z}(c)$ identificamos a $c$ con $0$, lo que obtenemos es una forma natural de ver a cada $\mathbb{Z}$-cadena como una copia isomorfa de $\mathbb{Z}$ en un modelo no estándar $\mathcal{M}$. Así, obtenemos la conslusión de que en un modelo no estándar dado $\mathcal{M}$, viven infinitas copias de los enteros, y más aún, es una cantidad densa de las mismas. 
\end{remark}

\section{El Principio combinatorio de Paris-Harrington, PH}

\subsection{Teoremas de Ramsey para particiones.}

\begin{definition}\rm
Sea $\sigma$ un cardinal, $[I]^{k}$ el conjunto de los subconjuntos de $I$ de cardinal $k$. Una función $P:[I]^{k}\to\sigma$ se llama una \textbf{partición} de $[I]^{k}$ en $\sigma$ partes.

\bigskip

\noindent Si $P:[I]^{k}\to\sigma$, llamamos $H\subset I$ \textbf{homogéneo} para $P$ si, y sólo si, $P$ es constante sobre $[H]^{n}$. Notaremos (siguiendo a Erdös) $\kappa\to(\lambda)_{\sigma}^{n}$ si siempre que $P:[\kappa]^{n}\to\sigma$, hay un $H\subset\kappa$ homogéneo para $P$ de cardinalidad $\lambda$.
\end{definition}

\begin{definition}\rm
Sea $H\subset\mathbb{N}$ finito. Se dice que $H$ es \textbf{relativamente grande} si $\mathrm{card}(H)\geq\min H$. Dados $n, r, k$ y $m$ números naturales, usaremos la notación 

\[
\xymatrix{
m\ar[r]_{*}&(k)_{r}^{n}\\
}
\]

\noindent para  indicar que para cualquier partición $P:[m]^{n}\to r$ hay un $H\subset m$ relativamente grande que es homogéneo para $P$ y de cardinalidad al menos $k$.
\end{definition}

\begin{definition}\rm
Un \textbf{árbol} es un conjunto parcialmente ordenado $(T, <_{T})$ tal que para todo $t\in T$, el conjunto $\widehat{t}:=\{s\in T:s<_{T}t\}$ está bien ordenado. Una \textbf{rama} de un árbol $T$ es una cadena (un subconjunto linealmente ordenado) máximal de $T$. Una \textbf{trayectoria } de $T$ es una cadena de $T$ que a su vez es un segmento inicial de $T$.  Un \textbf{árbol de ramaje finito} es un conjunto parcialmente ordenado $(T, <_{T})$ tal que:

\begin{itemize}
\item[i)] Existe $r\in T$ tal que $r<_{T} x$ para todo $x\in T$.
\item[ii)] Si $x\in T$, entonces $\{y:y<_{T}x\}$ es finito y linealmente ordenado por $<_{T}$.
\item[iii)] Si $x\in T$, cada conjunto finito (quizá vacío) $\{y_{1}, ..., y_{n}\}$ de elementos incomparables tal que cada $y_{i}>x$ y si $z>x$, entonces $z>y_{i}$ para algún $i$.
\end{itemize} 
\end{definition}

\noindent Por el Teorema de Enumeración, \emph{todo conjunto bien ordenado es isomorfo a algún único ordinal}. A este ordinal se le conoce como \textbf{tipo ordinal del conjunto}. La \textbf{altura} $\mathrm{Alt}(t)$ de $t$ en $T$ es el tipo de ordinal de $\widehat{t}$. El \textbf{nivel} $\alpha$ de $T$ es el conjunto $T_{\alpha}:=\{t\in T:\mathrm{Alt}(t)=\alpha\}$. La \textbf{ altura} de $T$ es $\min\{\alpha:T_{\alpha}=\emptyset\}$.

\begin{definition}\rm
Sea $\theta$ un ordinal y $\lambda$ un cardinal. Un árbol $T$ es un $(\theta, \lambda)$-\textbf{árbol} si:

\begin{itemize}
\item[i)] $(\forall\alpha<\theta)(T_{\alpha}\neq\emptyset)$.
\item[ii)] $T_{\theta}=\emptyset$.
\item[iii)] $(\forall\alpha<\theta)(\mathrm{card}(T_{\alpha})<\lambda)$.
\end{itemize}

\noindent Un $\aleph_{0}$-\textbf{árbol}, es simplemente un $(\aleph_{0}, \aleph_{0})$-árbol.
\end{definition}

\begin{lemma}[Lema de Köning] Todo $\aleph_{0}$-árbol tiene una rama cofinal, i.e., una rama que intercepta todos los niveles.
\end{lemma}

\begin{proof}
Sea $T$ un $\aleph_{0}$-árbol. Por inducción sobre $n<\omega$, elegimos $t_{n}\in T_{n}$ tal que $T^{t_{n}}$ es infinito y $t_{n}<_{T}t_{n+1}$. Entonces, $\{t_{n}:n<\omega\}$ es una rama cofinal de $T$.
\end{proof}

\begin{theorem}[Teorema Infinito de Ramsey] Para todo par de números naturales $n$ y $k$, se cumple que $\aleph_{0}\to(\aleph_{0})_{k}^{n}$. 
\end{theorem}

\begin{proof}
 Procedemos por inducción sobre $n$. Para $n=0$ no hay nada que probar, pues $f$ es constante sobre $[A]^{0}=\{\emptyset\}$. Sea $n>0$.
\bigskip

\noindent Definimos recursivamente una sucesión decreciente $A_{0}\supset A_{1}\supset ...$ de subconjuntos infinitos de $A$ y una sucesión $a_{0}, a_{1}, ...$ de elementos de $A$ con $a_{i}\in A_{j}$ sólo si $i\geq j$.
\bigskip

\noindent Comenzamos con $A_{0}=A$. Supongamos que ya se construyó $A_{i}$. Sea $a_{i}\in A_{i}$ arbitrario. Definimos $f_{i}:[A_{i}\setminus\{a-{i}\}]^{n-1}\to m$ mediante $f_{i}(b)=f(\{a_{i}\}\cup b)$. Como $A_{i+1}$, escogemos un subconjunto infinito $f_{i}$-homogéneo de $A_{i}\setminus\{a_{i}\}$.
\bigskip

\noindent Sea $m_{i}$ el valor que toma $f_{i}$ en $[A_{i+1}]^{n-1}$. Entonces, para cada $k<m$ el conjunto $B=\{a_{i}:m_{i}=k\}$ es $f$-homogéneo: cada subconjunto de $n$ elementos $c$ de $B$ tiene la forma $\{a_{i}\}\cup b$ para alguna $b\in [A_{i+1}]^{n-1}$. Se tiene $f(c)=f_{i}(b)=k$. Existe entonces una $k<m$ tal que $m_{i}=k$ para una cantidad infinita de $i\in\omega$. $B$ es infinito para esta $k$.
\end{proof}

\begin{theorem}[Teorema Finito de Ramsey] Para todos $k, n, m$ números naturales, existe un $l$ natural, tal que $l\to(m)_{k}^{n}$.
\bigskip
\end{theorem} 

\begin{proof}
Supóngase que no hay tal $l$. Para cada  $l<\omega$, sea 

\begin{small}
\begin{center}
$T_{l}:=\left\lbrace f:[l]^{n}\to k:\right.$ no existe un subconjunto de $l$ de tamaño $m$ homogéneo para $\left.f\right\rbrace$.
\end{center}
\end{small}

\noindent  Claramente, cada $T_{l}$ es finito, ya que hay finitas particiones para conjuntos finitos. Sea $f\in T_{l+1}$, luego hay un único $g\in T_{l}$ tal que $g\subset f$. Así, si ordenamos a 

\[
T=\bigcup_{l<\omega}T_{l}
\]

\noindent por inclusión, obtenemos un arbol finito. Cada $T_{j}\neq\emptyset$. Luego, obtenemos un arbol de ramaje finito. Por el Lema de Köning, podemos encontrar $f_{0}\subset f_{1}\subset ...$ con cada $f_{i}\in T_{i}$. 
\bigskip

\noindent Sea $f=\bigcup f_{i}$, entonces, $f:[\mathbb{N}]^{n}\to k$. Por el Teorema Infinito de Ramsey, hay un $X\subset \mathbb{N}$ infinito homogéneo para $f$. Sea $x_{1}, ..., x_{m}$, los primeros $m$ elementos de  $X$ y sea $s>x_{m}$, entonces, $\{x_{1}, ..., x_{m}\}$ es homogéneo para $f_{s}$. Contradicción. 
\end{proof}

\subsection{Paris-Harrington}

\begin{theorem}[Principio de Paris-Harrington] Para todos números naturales $n, r$ y $k$ hay un número natural $m$ tal que 

\[
\xymatrix{
m\ar[r]_{*}&(k)_{r}^{n}\\
}.
\]
\end{theorem}

\begin{proof}
Sean $n, r$ y $k$ números naturales. Supóngase que no existe tal $m$. Sea $P$ un contraejemplo para $m$. Si $P$ es una partición de $[m]^{n}$ en $r$ partes con ningún conjunto relativamente grande de tamaño a lo más $k$. Podemos ver el conjunto de contraejemplos como un arbol infinito de ramaje finito, es decir, si $P$ y $P'$ son contraejemplos para $m$ y $m'$ respectivamente, ponemos $P$ bajo $P'$ en nuestro arbol sólo si $m<m'$ y $P$ es una restricción de $P'$ a $[m]^{n}$.
\bigskip

\noindent Por el Lema de König hay un $P:[\omega]^{n}\to r$ tal que para todo $m$, la restricción de $P$ a $[m]^{n}$ es un contraejemplo para $m$. Por el Teorema Infinito de Ramsey, existe un $H\subset \omega$ infinito homogéneo para $P$. Pero entonces, al tomar $m$ suficientemente grande (comparado con $k$ y $\min H$) vemos que $H\cap m$ es, después de todo, un conjunto homogéneo relativamente grande para $P\upharpoonright [m]^{n}$ de tamaño a lo menos $k$.
\end{proof}

\noindent PH es una variante del Teorema Finito de Ramsey, pues en él sólo pedimos la condición adicional de que $H$ sea relativamente grande; y resulta deducible a partir del Teorema Infinito de Ramsey. Además PH es expresable en el lenguaje de PA, y sin embargo indemostrable en PA. Puede demostrarse que

\[
\forall n\in\mathbb{N}, \mathrm{PA}\vdash\forall r, k\exists m (\xymatrix{m\ar[r]_{*}&(k)^{n}_{r}}).
\]

\noindent Es decir, para cada número natural $n$ fijo, podemos formalizar la prueba de $\forall r, k\exists m (\xymatrix{m\ar[r]_{*}&(k)^{n}_{r}})$ en PA.

\subsubsection{Expresabilidad de PH en el lenguaje $\mathcal{L}_{A}$}  El Teorema Finito de Ramsey es una afirmación sobre números naturales, de esta forma, sería más preciso demostrarle sin recurrir a métodos infinitarios \cite{maker}. Sin embargo, estos métodos requieren presentarse con un enfoque diferente al que aquí planteamos. El lector interesado en una formulación meramente finitaría puede consultar \cite{graham}. 
\medskip

\noindent PH es expresable en el Lenguaje de la Aritmética de forma natural. Podemos ver tal factor de dos maneras posibles. Primeramente, debido a que PA es equivalente al marco axiomatico ZF para la Teoría de Conjuntos, si reemplazamos en ZF el Axioma de Infinitud por su negación. De esta manera, se obtiene que PH es formalizable en PA sin necesidad de ningún tipo de codificación \cite{paris}. Por otra parte, todas las nociones sobre conjuntos y particiones (finitarias) son expresables en el lenguaje de PA usando códigos de Gödel. El razonamiento es el siguiente: hay fórmulas $S(u)$, $l(u, v)$ y $e(v, u, i)$ en el lenguaje $\mathcal{L}_{A}$ tales que en $\mathbb{N}$, el modelo natural de la aritmética, $S(u)$ define el conjunto de códigos para secuencias finitas, $l(u, v)$ dice que $u$ es el código de una secuencia de longitud $v$, y $e(v, u, i)$ dice que $v$ es el $i$-ésimo elemento codificado por $u$ \cite{maker}. De esta manera, a cada conjunto finito le asociamos una secuencia finita, de modo que la longitud de esta secuencia, será el correspondiente cardinal del conjunto en cuestión. Codificamos las secuencias con la función $\beta$ de Gödel \cite{tossillo}. 
\medskip

\noindent En primera instancia, fijamos el tamaño de los conjuntos. Ahora, pedimos que todos los conjuntos de cierto tamaño tengan una propiedad dada. De esta manera, según expusimos antes, podemos decir: todo código que represente una secuencia de longitud de cierto tamaño cumple cierta propiedad. Así, acotamos los tamaños de los conjuntos dentro de los cuantificadores. Resulta entonces que todas las propiedades de los conjuntos se pueden traducir a los códigos de Gödel de tales conjuntos. 
\medskip

\noindent Para codificar particiones de conjuntos finitos procedemos así: sean $m, n, c\in\mathbb{N}$ con $n\leq m$ y pensemos en ellos como conjuntos finitos. Escribimos `$H\in\mathrm{Part}([m]^{n}, c)$' para índicar que `$H$ es una partición de $[m]^{n}$ en $c$ partes'. daremos una expresión de $\mathcal{L}_{A}$ para esto. Notése que cualquier elemento en  $[m]^{n}$ puede verse como $\{m-i_{0}, ..., m-i_{n-1}\}$ con $1\leq i_{j}\leq m$ para $0\leq j\leq n-1$. Ahora, $[m]^{n}$ tiene exactamente $\binom{m}{n}=\frac{m!}{n!(m-n)!}$ elementos. Luego, los podemos enlistar como

\[
[m]^{n}=m_{1}, ..., m_{\binom{m}{n}}.
\]

\noindent Existe sólo un número finito de particiones de conjuntos finitos\footnote{Pues, observése que una partición $P:[I]^{k}\to \sigma$ es simplemente un subconjutno del producto cartesiano $[I]^{k}\times\sigma$ y, al ser $[I]^{k}, \sigma$ finitos, necesariamente $[I]^{k}\times\sigma$ es finito, de modo que sólo hay finitas particiones de $[I]^{k}$ en $\sigma$ partes. }. Sea $H\in\mathrm{Part}([m]^{n}, c)$, cada partición  es un conjunto de pares, digamos

\[
H:=\left\{(m_{\delta}, c_{\alpha}):1\leq\delta\leq\binom{m}{n}\wedge 0\leq\alpha\leq\ c\right\}
\]

\noindent Así, para codificar  $H$ debemos seguir los siguientes tres simples pasos: 

\begin{itemize}
\item[1.] Codifique sobre cada conjunto $m_{\delta}$, para $1\leq \delta\leq \binom{m}{n}$ de forma natural (según se expuso en Capítulo 1, Sección 3.).
\item[2.] Codifique sobre cada par $(m_{\delta}, c_{\alpha})$, $1\leq\delta\leq\binom{m}{n}\wedge 0\leq\alpha\leq\ c$ con la función par de Gödel (Capítulo 1, Sección 3.).
\item[3.] Codifique sobre el conjunto obtenido en el Paso 2, según se hizo en el Paso 1.
\end{itemize}

\noindent Procedamos con el Paso 1. Sea $m_{\delta}\in[m]^{n}$, para $1\leq \delta\leq \binom{m}{n}$. Luego:

\begin{eqnarray*}
\#(m_{\delta})&=&\#(\{m-i_{0}, ..., m-i_{n-1}\})\\
&=&\prod_{\substack{1\leq_{i_{j}}\leq m\\0\leq j\leq n-1}}p_{i_{j}}^{m-i_{j}}
\end{eqnarray*}

\noindent para ciertos $1\leq i_{0}, ..., i_{n-1}\leq m$. Así, para cada par $(m_{\delta}, c_{\alpha})$ tenemos un único par correspondiente 

\[
\left(\prod_{\substack{1\leq_{i_{j}}\leq m\\0\leq j\leq n-1}}p_{i_{j}}^{m-i_{j}}, c_{\alpha} \right)
\]

\noindent Y para éste, procedemos con el Paso número 2. Así, obtenemos el conjunto

\[
\left\{\left(\prod_{\substack{1\leq_{i_{j}}\leq m\\0\leq j\leq n-1}}p_{i_{j}}^{m-i_{j}}, c_{\alpha} \right):0\leq\alpha\leq c-1\right\}.
\]

\noindent Luego,  

\begin{eqnarray*}
\#(m_{\delta}, c_{\alpha})&:=&\left\langle\#(m_{\delta}), c_{\alpha}\right\rangle\\
&=&\left\langle\#(\{m-i_{0}, ..., m-i_{n-1}\}), c_{\alpha}\right\rangle\\
&=&\left\langle\prod_{\substack{1\leq_{i_{j}}\leq m\\0\leq j\leq n-1}}p_{i_{j}}^{m-i_{j}}, c_{\alpha} \right\rangle\\
&=&\frac{\left[\displaystyle\prod_{\substack{1\leq_{i_{j}}\leq m\\0\leq j\leq n-1}}p_{i_{j}}^{m-i_{j}}+c_{\alpha}+1\right]\left[\displaystyle\prod_{\substack{1\leq_{i_{j}}\leq m\\0\leq j\leq n-1}}p_{i_{j}}^{m-i_{j}}+c_{\alpha}\right]}{2}+c_{\alpha} 
\end{eqnarray*}

\noindent Y, obtenemos el conjunto

\[
\left\{
\frac{\left[\displaystyle\prod_{\substack{1\leq_{i_{j}}\leq m\\0\leq j\leq n-1}}p_{i_{j}}^{m-i_{j}}+c_{\alpha}+1\right]\left[\displaystyle\prod_{\substack{1\leq_{i_{j}}\leq m\\0\leq j\leq n-1}}p_{i_{j}}^{m-i_{j}}+c_{\alpha}\right]}{2}+c_{\alpha} :0\leq\alpha\leq c-1
\right\}.
\]

\noindent Finalmente completamos el Paso 3, poniendo  

\begin{equation*}
\#(H):=\prod_{\substack{0\leq r\leq \binom{m}{n}\\0\leq \alpha\leq c-1}}p_{r}^{\frac{\left[\displaystyle\prod_{\substack{1\leq_{i_{j}}\leq m\\0\leq j\leq n-1}}p_{i_{j}}^{m-i_{j}}+c_{\alpha}+1\right]\left[\displaystyle\prod_{\substack{1\leq_{i_{j}}\leq m\\0\leq j\leq n-1}}p_{i_{j}}^{m-i_{j}}+c_{\alpha}\right]}{2}+c_{\alpha}}
\end{equation*}

\noindent\textbf{PH.}Para cualquier números naturales $b, e, y$ codificando secuencias de longitud $n, c$ y al menos   $\lambda$, respectivamente, existe un número natural $a$ codificando una secuencia de longitud $m$ tal que $\xymatrix{m\ar[r]_{*}&(k)_{r}^{n}\\}$, i.e., para cualquier $H\in\mathrm{Part}([m]^{n}, c)$, existe un único número natural $\beta\leq\mathrm{Long}(e)$ tal que para cualquier $1\leq k_{1}\leq\binom{m}{n}$, $(m_{k_{1}}, c_{\beta})\in H$ y no existe $\alpha\in\mathbb{N}$ tal que para algún $0\leq k_{2}\leq\binom{m}{n}$, $(m_{k_{2}}, c_{\alpha})\in H$  siempre que $b$ es tal que existe una natural $s$ codificando una secuencia tal que $y=b*s$ con $\mathrm{Long}(y)\geq(y)_{0}$. 

\begin{remark}\rm En la formulación anterior, usamos  $H$ como un conjunto de pares, pero el lector no debe olvidar, que $H$ también puede ser visto como un número natural. En este trabajo  conseguimos la sentencia siguiente para el Teorema Finito de Ramsey, en el cual tratamos la noción se secuencia en lugar de la de conjunto, y la noción de longitud de una secuencia en lugar de la de cardinal.
\end{remark}

\begin{center}
\fbox{
\begin{minipage}[b][1.05\height]
[t]{1\textwidth} 
$(\forall b\in\mathrm{Seq})(\forall e\in\mathrm{Seq})(\forall y\in\mathrm{Seq})(\exists a\in\mathrm{ Seq})\{[\mathrm{Long}(b)=n\wedge\mathrm{Long}(a)=m\wedge\mathrm{Long}(e)=c\wedge\mathrm{Long}(y)\leq\lambda\wedge n\leq m]:(\forall H\in\mathrm{Part}([m]^{n}, c))[(\exists s\in\mathrm{Seq})[y=b*s]\to(\exists!\beta\leq\mathrm{Long}(e)\in\mathbb{N}))[(\forall k_{1}\leq\binom{m}{n}\in\mathbb{N})[(m_{k_{1}}, c_{\beta})\in H]\wedge(\neg\exists\alpha\leq\mathrm{Long}(e)\in\mathbb{N})[(\exists k_{2}\leq\binom{m}{n}\in\mathbb{N})[(m_{k_{2}}, c_{\alpha})\in H]]]] \wedge \mathrm{Long}(y)\geq(y)_{0}\} $
\end{minipage}}
\end{center}

\noindent  Teniendo en cuenta los significados de Seq, Long y que todo número $i$ en $\mathbb{N}$ se puede escribir como $\underbrace{1+1+1+...+1}_{i-\mathrm{veces}}$, y definimos $a*b=a\cdot\prod_{x\leq\mathrm{Long}(b)}p_{\mathrm{Long}(a)+x+1}^{(b)_{x}+1}$, si $a, b\neq1$, y acordando que $\mu x\leq g[p_{x}|g\wedge\neg(p_{x+1}|g)])=0$ cuando $g\not\in$Seq o $g=1$, tenemos la siguiente sentencia del lenguaje de la Aritmética para PH.

\begin{center}
\fbox{
\begin{minipage}[b][1.05\height]%
[t]{1\textwidth} 
$
 [(\forall b(b=1\vee(b>1\wedge\forall x_{1}\leq b[p_{x_{1}+1}|b\to p_{x_{1}}|b])))
  (\forall e(e=1\vee(e>1\wedge\forall x_{2}\leq e[p_{x_{2}+1}|e\to p_{x_{2}}|e])))
  (\forall y(y=1\vee(y>1\wedge\forall x_{3}\leq y[p_{x_{3}+1}|y\to p_{x_{3}}|y])))
  (\exists a(a=1\vee(a>1\wedge\forall x_{4}\leq a[p_{x_{4}+1}|a\to p_{x_{4}}|a])))]\\
 \left\{\begin{array}{c}
 \\
  \\
\end{array}\right
 [
 \mu x_{5}\leq b[p_{x_{5}}|b\wedge\neg(p_{x_{5}+1}|b)])=n\wedge\\
 \mu x_{6}\leq a[p_{x_{5}}|a\wedge\neg(p_{x_{5}+1}|a)])=m\wedge
 \mu x_{7}\leq e[p_{x_{5}}|e\wedge\neg(p_{x_{5}+1}|e)])=c\wedge\\
 \mu x_{8}\leq y[p_{x_{5}}|y\wedge\neg(p_{x_{5}+1}|y)])\leq\lambda\wedge n\leq m
 ]:\\
 \left[
 \forall H=\displaystyle\prod_{\substack{0\leq r\leq n-1\\0\leq \alpha\leq c-1}}p_{r}^{\frac{\left[\displaystyle\prod_{\substack{1\leq_{i_{j}}\leq m\\0\leq j\leq n-1}}p_{i_{j}}^{m-i_{j}}+c_{\alpha}+1\right]\left[\displaystyle\prod_{\substack{1\leq_{i_{j}}\leq m\\0\leq j\leq n-1}}p_{i_{j}}^{m-i_{j}}+c_{\alpha}\right]}{2}+c_{\alpha}}
 \right]\\
 (
 \exists s(s=1\vee(s>1\wedge\forall x_{6}\leq s[p_{x_{6}+1}|s\to p_{x_{6}}|s]))\\
\left(y=b\displaystyle\prod_{x_{7}\leq\mu x_{8}\leq s[p_{x_{8}}|s\wedge\neg(p_{x_{8}+1}|s)]}p_{\mu x_{5}\leq y[p_{x_{5}}|y\wedge\neg(p_{x_{5}+1}|y)]+x_{7}+1}^{\mu x_{9}\leq x_{7}+1[p_{x_{7}}^{x_{9}+1}|s\wedge\neg(p_{x_{7}}^{x_{9}+1}|s)]+1}\right)\to\\
(\exists!\beta=\underbrace{1+\cdots+1}_{\beta-\mathrm{veces}}\leq\mu x_{7}\leq e[p_{x_{7}}|e\wedge\neg(p_{x_{7}+1}|e)]))\\
\left[\left(\forall k_{1}=\underbrace{1+\cdots+1}_{k_{1}-\mathrm{veces}}\leq\\
\frac{m(m-1)(m-2)\cdots(m-(m-1))}{n(n-1)\cdots(n-(n-1))(m-n)(m-n-1)\cdots((m-n-(m-n-1)))}=h\right)\right.\\\
\exists\eta=\underbrace{1+\cdots+1}_{\eta-\mathrm{veces}}\leq h
\left[p^{\#(\#(m_{k_{1}}), c_{\beta}) }_{\eta}|H \right]\wedge\\(\neg\exists \alpha=\underbrace{1+\cdots+1}_{a-\mathrm{veces}}\leq\mu x_{7}\leq e[p_{x_{5}}|e\wedge\neg(p_{x_{5}+1}|e)]))\\
\left[\left(\exists k_{2}=\underbrace{1+\cdots+1}_{k_{2}-\mathrm{veces}}\leq\frac{m(m-1)(m-2)\cdots(m-(m-1))}{n(n-1)\cdots(n-(n-1))(m-n)(m-n-1)\cdots((m-n-(m-n-1)))}\right)\right.\\
\exists\rho=\underbrace{1+\cdots+1}_{\rho-\mathrm{veces}}\leq h
\left.\left.\left.\left(p_{\rho}^{\#(\#(m_{k_{2}}), c_{\alpha})}|H\right)\right]\right]\right]\wedge\\
\mu x_{8}\leq y[p_{x_{5}}|y\wedge\neg(p_{x_{5}+1}|y)])\geq\mu x_{10}\leq 0+1[p_{x_{10}}^{0+1}|y\wedge \neg(p_{x_{10}}^{0+2}|y)]
\left.\begin{array}{c}
 \\
  \\
\end{array}\right\}
$
\end{minipage}}
\end{center}

\noindent Aquí, $\#(\#(m_{k_{2}}), c_{\alpha})$ se interpreta de forma obvia, y ha sido escrito de esta forma breve para mayor comodidad en la lectura y notación. Omitiendo  $\mu x_{8}\leq y[p_{x_{5}}|y\wedge\neg(p_{x_{5}+1}|y)])\geq\mu x_{10}\leq 0+1[p_{x_{10}}^{0+1}|y\wedge \neg(p_{x_{10}}^{0+2}|y)]$ tendríamos  el Teorema Finito de Ramsey.

\section{La independencia de PH en PA}

\noindent Aquí presentamos la prueba de independencia de PH sobre PA brindada por Jeff Paris y Leo Harrington en \cite{paris}.  Definimos una cierta teoría $T$, y demostraremos que sobre PA puede probarse que $\mathrm{Con}(T)\to\mathrm{Con}(PA)$. Y concluiremos al probar que PH implica $\mathrm{Con}(T)$ es también un teorema de PA. Expandimos el lenguaje $\mathcal{L}_{A}$ agregandole una colección infinita numerable de constantes $c_{0}, c_{1}, ...$. Sea $T$, la teoría engendrada en el seno de los axiomas:

\begin{itemize}
\item[i)] Las ecuaciones recursivas usuales definidas para $+, \cdot, <$ más los axiomas de inducción sólo para fórmulas límitadas.
\item[ii)] Para cada $i=0, 1, ...$, el axioma $(c_{i})^{2}<c_{i+1}$.
\item[iii)] Para cada subconjunto finito $i=i_{1}, ..., i_{r}$ de $\omega$, sea $c(i)=c_{i_{1}}, ..., c_{i_{r}}$. Para cada $i<k, k'$ y cada sentencia $\psi(y, z)$ (donde $k, k'$ y $z$ tienen todas el mismo tamaño) tenemos el axioma:

\[
\forall y<c_{i}[\psi(y; c(k))\leftrightarrow\psi(y; c(k'))].
\]
\end{itemize}

\begin{proposition}
$\mathrm{Con}(T)$ implica $\mathrm{Con}(PA)$.
\end{proposition}

\begin{proof}
Sea $\mathcal{U}$ un modelo de $T$ e $I$ el segmento inicial de $\mathcal{U}$ del cual $a<c_{i}$ para algún $i\in\omega$. Por (ii), $I$ es cerrado bajo $+, \cdot$. Entonces, será suficiente mostrar los dos hechos siguientes:

\begin{itemize}
\item[A1.] $\mathcal{I}=(I, +, \cdot, <)\models PA$.
\item[A2.] Dados $i<k$, $a<c_{i}$ y $\theta(y)$, donde $k, a$ y $y$ son todos de longitud adecuada, 
\begin{center}
$\mathcal{I}\models\theta(a)$ si, y sólo si, $\mathcal{U}\models\theta^{*}(a; c(k))$.
\end{center}

\noindent Procedemos por inducción sobre $\theta$. Supóngase que $\theta$ es 

$$\exists x<z_{1}\psi^{*}(x, y, z_{2}, ..., z_{r}).$$

\noindent  Así, $\mathcal{I}\models\theta(a)$ si, y sólo si, para algún $b$ en $I$ y algún $j$ (con $\min(j)$ grande), $\mathcal{U}\models\psi^{*}(b, a, c(j))$, lo que ocurre si, y sólo si, para algún $k'$ (de nuevo con $\min(k')$ grande), $\mathcal{U}\models\theta^{*}(a, c(k'))$ lo que ocurre si, y sólo si $\mathcal{U}\models\theta^{*}(a; c(k))$.
\end{itemize}

\noindent Ahora bien, A1., es consecuencia de A2. En efecto, por (i), para toda $\theta$, $U$ satisface la inducción para $\theta^{*}$.
\end{proof}

\begin{proposition}
PH implica $\mathrm{Con}(T)$.
\end{proposition}

\noindent Por el segundo Teorema de Incompletitud, será suficiente demostrar que la proposición se puede probar en PA, para obtener la independencia de PH sobre PA. Pues  $PA$ no puede demostrar $\mathrm{Con}(PA)$ luego, PA no puede demostrar PH. Ahora, necesitamos algunos lemas.
\medskip

\begin{lemma}
Sean $P_{0}$ y $P_{1}$ particiones de $[M]^{e}$ en $r_{0}$ y $r_{1}$ partes, respectivamente. Entonces, hay una partición $P$ de $[M]^{e}$ en $r_{0}\cdot r_{1}$ partes, tal que para $H\subset M$, $H$ es homogéneo para $P$ si, y sólo si, $H$ es homogéneo para  ambas $P_{0}$ y $P_{1}$.
\end{lemma}

\begin{proof}
Basta considerar $P(a)=(P_{0}(a), P_{1}(a))$.
\end{proof}

\begin{lemma}
Un conjunto $H$ es homogéneo para una partición $P$ de $[M]^{e}$ si, y sólo si, todo subconjunto de $H$ de tamaño $e+1$ es homogéneo para $P$.
\end{lemma}

\begin{proof}
Sea $a=a_{1}, ..., a_{e}$ los primeros $e$ elementos de $H$. Tomése $b=b_{1}, ..., b_{e}$ tal que $P(a)\neq P(b)$ y tal que $b_{1}+...+b_{e}$ es minimizado. Si $i$ es el menor índice tal que $a_{i}\neq b_{i}$, entonces, $\{a_{1}, ..., a_{i}, b_{i}, ..., b_{e}\}$ no es homogéneo y de tamaño $e+1$.
\end{proof}

\begin{definition}\rm
Definimos $\surd r$ el primer número natural $s$ tal que $s^{2}\geq r$. Observese que para muchos $r$ (i.e., $r\geq 7$), $r\geq 1+2\surd r$.
\end{definition}

\begin{lemma}
Dada $P:[M]^{e}\to r$ hay un $P':[M]^{e+1}\to(1+2\surd r)$ tal que para todo $H\subset M$ de cardinal mayor que $e+1$, $H$ es homogéneo para $P$ si, y sólo si, $H$ es homogéneo para $P'$.
\end{lemma}

\begin{proof}
Sea $s=\surd r$. Definanse funciones $Q$ (para cociente) y $R$ (para residuo) ambas mapeando a $[M]^{e}$ en $s$ según la ecuación $P(a):=sQ(a)+R(a)$. Para $b=b_{1}, ..., b_{e}, b_{e+1}$ en $[M]^{e+1}$, sea $b'=b_{1}, ..., b_{e}$. Ahora, definimos $P'$ sobre $[M]^{e+1}$ por:

\[
P'(b):=\left\{\begin{array}{ccc}
0&\mathrm{si}&b\textsl{ es homogéneo para }P,\\
(0, R(b'))&\mathrm{si}&b\textsl{ es homogéneo para }Q\textsl{ y no para P}\\
(1, Q(b'))& &\textsl{ en otro caso.}
\end{array}
\right.
\]

\noindent Sea $H$ homogéneo para $P'$ de cardinalidad $>e+1$, y sea $c$ los primeros $e+1$ miembros de $H$. Debemos ver que $P'(c)=0$ para verificar que $H$ es homogéneo  para $P$, por el Lema 5.4. Notése que para cada $a$ en $[c]^{e}$ hay un $b$ en $[H]^{e+1}$ tal que $b'=a$. Supóngase que $P'(c)=(1, i)$. Entonces, por las observaciones previas, $Q(a)=i$ para toda $a$ en $[c]^{e}$ así que $c$ es homogénea para $Q$, contradiciendo la definición de $P'$. Así, supóngase que $P'(c)=(0, j)$ así que $c$ es $Q$ digamos $Q(a)=i$ para todo $a$ en $[c]^{e}$. Pero entonces $P(a)=si+i$ para toda tal $a$ así que $c$ es homogéneo para $P$, de nuevo contradiciendo la definición de $P'$.
\end{proof}

\begin{lemma}
Supóngase que nos son dadas $n$ particiones $P_{i}:[M]^{e_{i}}\to r_{i}$, $i<n\leq n$. Sea $e=\max_{i}e_{i}$ y $r=\prod_{i}\max (r_{i}, 7)$. Hay una partición $P:[M]^{e}\to r$ tal que para todo $H\subset M$ de cardinalidad mayor que $e$, $H$ es homogéneo para $P$ si, y sólo si, $H$ es homogéneo para todos los $P_{i}$'s. 
\end{lemma}

\begin{proposition}
Para todos $e, r, k$ hay un $M$ tal que para cualquier familia $( P_{\epsilon}$, $\epsilon<2^{M})$ de particiones $P_{\epsilon}:[M]^{e}\to r$, hay un $X$ de cardinalidad mayor o igual que $k$ tal que:

\begin{itemize}
\item[i)] Si $a, b\in X$ y $a<b$, entonces $a^{2}<b$,
\item[ii)] Si $a\in X$ y $\epsilon<2^{a}$, entonces $X\sim(a+1)$ es homogéneo para $P_{\epsilon}$.
\end{itemize}
\end{proposition}

\begin{statement}\rm
La proposición anterior implica $\mathrm{Con}(T)$.
\end{statement}

\begin{proof}
Ver \cite{paris}.
\end{proof}

\noindent  Para cualquier función $g$, sea $g^{(x)}$, $g$ compuesta con sí misma $n$-veces. Sea $f_{0}(x)=x+2$ y sea $f_{n+1}(x)=f_{n}^{(x)}(2)$. Se puede observar que $f_{1}(x)\geq 2x$, $f_{2}(x)\geq 2^{x}$, $f_{3}(x)\geq\beth_{x}$ donde $\sqsupset\beth_{x}=2^{2^{\ddots^{2}}}$ y así para $f_{4}, f_{5}, ...$ 

\begin{lemma} Lo siguiente se cumple.
\begin{itemize}
\item[i)]Para todo $P$ hay un $Q:[M]^{1}\to p+1$ tal que si $X$ es homogéneo para $Q$ y de cardinalidad al menos 2, entonces $\min(X)\geq p$.
\item[ii)]Para cada $m$ hay una partición $R:[M]^{2}\to r$ (donde $r$ depende solo de $m$) tal que si $X\subset M$ es relativamente grande y homogéneo para $R$ y de cardinalidad mayor que 2, entonces para todo $x, y\in X$, $x<y$ se tiene que $f_{m}(x)<y$.
\end{itemize}
\end{lemma}

\begin{proof}
Ver \cite{paris}.
\end{proof}

\begin{lemma}
Sea $P:[M]^{e}\to s$ ($e\geq2$) y $m$ dado. Hay una partición $P^{*}:[M]^{e}\to s'$, donde $s'$ depende solo de $m, e$ y $s$, tal que si hay un $Y\in M$ relativamente grande y homogéneo para $P^{*}$ de cardinalidad mayor que $e$, enotnces hay un $X\subset M$ tal que $X$ es homogéneo para $P$ y $\mathrm{card}(X)$ es a lo menos $e+1$ y $f_{m}(\min(X))$.
\end{lemma}

\begin{proposition}
PH implica  5.8.
\end{proposition}

\begin{proof}
Nos son dados $e, r,k$ y debemos construir $M$ como en 5.1.8. Encuentre un $p$ tal que para todo $a\geq p$, $f_{3}(a)$ es razonablemente grande en comparación con $e, r, k$ y $a$. Notése que $f_{3}(y)\geq\beth_{y}$. Sea $e'=2e+1$. Ahora, dado cualquier $M$ y cualquier familia $P_{\epsilon}:[M]^{e}\to r$ para $\epsilon<2^{M}$, define un nuevo $S:[M]^{e'}\to 2$ por $S(a, b, c)=0$ si $P_{\epsilon}(b)=P_{\epsilon}(c)$ para todo $\epsilon<2^{a}$. $S(a, b, c)=1$ en otro caso. Sea $Q$ como en 5.10. y como en 5.11. para $m=1$. Use 5.7. para combinar $Q, R$ y $S$ en $P$ y luego use 5.12. para obtener $P^{*}:[M]^{e'}\to s'$. El número $s'$ depende solo de $e'$ y de $p$. Ahora, aplicamos el principio PH. Encontramos un $M$ tal que $\xymatrix{M\ar[r]_{*}&(e'+1)_{s'}^{e'}}$. Por 5.1.10. hay un $M$ tal que es homogéneo para $Q, R, S$ ($\min(X)\geq p$) con $\mathrm{card}(X)\geq f_{3}(\min(X))$. Ya que $X$ es homogéneo para $R$, y ya que $f_{2}(y)\geq y^{2}$ para aquellos $y$ lo suficientemente grandes para estár en $X$, $X$ satisface ii) de 5.8.
\medskip

\noindent Para verificar iii) de 5.8. reemplazamos $X$ por $X\sim d=X'$ donde $d=d_{1}, ..., d_{e}$ son los  últimos $e$ elementos de $X$. Sea $d_{\epsilon}'=P_{\epsilon}(d)$. Si mostramos que para todo $a<b_{1}<...<b_{e}$ en $X'$ y todo $\epsilon<2^{a}$, $P_{\epsilon}(b_{i})=i_{\epsilon}$. Es suficiente con mostrar que $S(a, b, c)=0$, para algún (y así, por homogéneidad, para todo) $1+2e$ túpla $a, b, c$ de $X$. Sea $a=\min(X)$ y considerese $e$-úplas consecutivas de $X\sim(a+1)$. Nuestra anterior elección de $p$ debe ser tal que haya más de $r^{(2^{a})}$ $e$-úplas para entonces podemos encontrar $e$-úplas para todo $\epsilon<2^{a}$, como se deseaba.
\end{proof}

\subsection{Refinamientos} \hspace{0.1cm}\\

\noindent  Considerése $\mathrm{Rfn}_{\Sigma_{1}}$ la afirmación de la teoría de números que dice que para toda $\Sigma_{1}$-sentencia $\psi$, si PA$\vdash\psi$ entonces $\psi$. 

\begin{theorem}
Es un teorema de la Aritmética de Peano que PH equivale a $\mathrm{Rfn}_{\Sigma_{1}}$.
\end{theorem}

\begin{proof}
Se sabe que para todos $e, r, k$, PA$\vdash\exists M(\xymatrix{M\ar[r]_{*}&(k)_{r}^{e}})$. Este hecho en sí mismo es un teorema de PA. Una aplicación de $\mathrm{Rfn}_{\Sigma_{1}}$ nos da PH.
\medskip

\noindent Asumáse PH. Probemos $\mathrm{Rfn}_{\Sigma_{1}}$. Sea $\psi$ una $\Sigma_{1}$-sentencia. Probemos que si $\neg\psi$, entonces $\mathrm{Con}(\mathrm{PA}+\neg\psi)$. Si $\psi$ es falsa en $\omega$, entonces $\mathrm{Con}(\mathrm{T}+\neg\psi
)$, usando PH pero la prueba de 5.11. muestra que $\mathrm{Con}(\mathrm{T}+\neg\psi)$ implica $\mathrm{Con}(\mathrm{PA}+\neg\psi)$.
\end{proof}

\noindent Defina una función $f$ recursiva por  $f(e)=$ el menor $M$ tal que $\xymatrix{M\ar[r]_{*}&(e+1)_{e}^{e}}$. 

\begin{theorem}
Si $g$ es una (descripción de $a$) función recursiva y si además PA$\vdash$ ``$g$ es total'', entonces, para todo $e$ suficientemente grande, $f(e)>g(e)$.
\end{theorem}

\begin{proof}
 Sea $S$ un subconjunto finito de $T$ y sea $c_{0}, c_{1}, ..., c_{k-1}$ las constantes que inciden en $S$. Podemos interpretar $c_{0}, ..., c_{k-1}$ usando miembros del intervalo $(e, f(e))$. Si $g(e)<f(e)$ para una cantidad infinita grande de $e$, lo anterior muestra la consistencia de $T$ más los siguientes axiomas en una nueva constante $e$:

\begin{center}
$e<c_{0}$; $\neg\exists x\leq c_{i}(g(e)=x)$ para todo $i\leq\omega$.
\end{center}

\noindent Obtenemos la consistencia de PA$+\exists e(g(e)$ no está definida).
\end{proof}

\end{document}